\documentclass[reqno]{amsart}
\usepackage{enumerate, amscd, amssymb, mathtools}

\allowdisplaybreaks
\usepackage{color}

\definecolor{darkgreen}{rgb}{0,0.6,0}
\definecolor{darkblue}{rgb}{0,.2,.7}
 \usepackage[%
   colorlinks, 
   citecolor=darkgreen, linkcolor=darkblue,
   pdfpagelabels=true,
   unicode=true,
   pdfusetitle
 ]{hyperref}

\newcommand\II{\mathrm{I}\hskip-.3mm\mathrm{I}}

\makeatletter
\newcommand{\doublewidetilde}[1]{{%
  \mathpalette\double@widetilde{#1}%
}}
\newcommand{\double@widetilde}[2]{%
  \sbox\z@{$\m@th#1\widetilde{#2}$}%
  \ht\z@=.9\ht\z@
  \widetilde{\box\z@}%
}
\makeatother

\newcommand{\dist}{\operatorname{dist}}
\newcommand\seq{=}
\newcommand\define{\mathrel{:= }}
\newcommand\ede{\define}

\newcommand{\CC}{\mathbb C}

\newcommand{\NN}{\mathbb N}

\newcommand{\RR}{\mathbb R}

\newcommand{\ZZ}{\mathbb Z}
\newcommand{\maX}{\mathcal X}

\newcommand{\CI}{{\mathcal C}^{\infty}}
\newcommand{\CIc}{{\mathcal C}^{\infty}_{\text{c}}}
\newcommand\pa{\partial}

\newcommand{\End}{\operatorname{End}}

\newcommand{\vol}{\operatorname{vol}}
\newcommand{\dvol}{\,\operatorname{dvol}}

\newcommand{\maD}{\mathcal D}

\newcommand{\maF}{\mathcal F}

\newcommand{\maW}{\mathcal W}

\newcommand{\rinj}{\mathop{r_{\mathrm{inj}}}}
\newcommand\<{\langle}
\renewcommand\>{\rangle} 

\newcommand\cdotH[1]{\overset{\bullet}{H}{}^{#1}}

\newtheorem{theorem}{Theorem}[section]
\newtheorem{proposition}[theorem]{Proposition}
\newtheorem{corollary}[theorem]{Corollary}

\newtheorem{assumption}[theorem]{Assumption}
\newtheorem{lemma}[theorem]{Lemma}
\newtheorem{notation}[theorem]{Notations}
\theoremstyle{definition}
\newtheorem{definition}[theorem]{Definition}
\theoremstyle{remark}
\newtheorem{remark}[theorem]{Remark}
\newtheorem{example}[theorem]{Example}

\author[N. Gro{\ss}e]{Nadine Gro{\ss}e} \address{Mathematisches
  Institut, Universit\"at Freiburg, 79104 Freiburg, Germany}
\email{nadine.grosse@math.uni-freiburg.de}

\author[V. Nistor]{Victor Nistor} \address{Universit\'{e} de Lorraine,
  UFR MIM, Ile du Saulcy, CS 50128, 57045 METZ, France and
  Inst. Math. Romanian Acad.  PO BOX 1-764, 014700 Bucharest Romania}
\email{victor.nistor@univ-lorraine.fr}

\thanks{N.G. has been partially supported by SPP 2026 (Geometry at
  infinity), funded by the DFG. V.N. has been partially supported by
  ANR-14-CE25-0012-01.}

\begin{document}

\title[Mixed problems]{Uniform Shapiro-Lopatinski conditions and
  boundary value problems on manifolds with bounded geometry}

\begin{abstract} 
We study the regularity of the solutions of second order boundary
value problems on \emph{manifolds with boundary and bounded geometry}.
We first show that the regularity property of a given boundary value
problem $(P, C)$ is equivalent to the \emph{uniform} regularity of the
natural family $(P_x, C_x)$ of associated boundary value problems in
local coordinates. We verify that this property is satisfied for the
Dirichlet boundary conditions and strongly elliptic operators via a
compactness argument. We then introduce a \emph{uniform
  Shapiro-Lopatinski regularity condition,} which is a modification of
the classical one, and we prove that it characterizes the boundary
value problems that satisfy the usual regularity property. We also
show that the natural Robin boundary conditions always satisfy the
uniform Shapiro-Lopatinski regularity condition, provided that our
operator satisfies the strong Legendre condition. This is achieved by
proving that ``well-posedness implies regularity'' via a modification
of the classical ``Nirenberg trick''. When combining our regularity
results with the Poincar\'e inequality of (Ammann-Gro{\ss}e-Nistor,
preprint 2015), one obtains the usual well-posedness results for the
classical boundary value problems in the usual scale of Sobolev
spaces, thus extending these important, well-known theorems from
smooth, bounded domains, to manifolds with boundary and bounded
geometry. As we show in several examples, these results do not hold
true anymore if one drops the bounded geometry assumption. We also
introduce a \emph{uniform Agmon condition} and show that it is
equivalent to the coerciveness. Consequently, we prove a
well-posedness result for parabolic equations whose elliptic generator
satisfies the uniform Agmon condition.
\end{abstract}

\maketitle

\tableofcontents
  
\section{Introduction}

We study the regularity of the solutions of general second order
boundary value problems on \emph{manifolds with boundary and bounded
  geometry.} We obtain several sufficient and several necessary
conditions for the regularity of the solution in the usual scale of
Sobolev spaces. To state these results in more detail, we need to
introduce some notation.

One of the reasons for our interest in manifolds with bounded geometry
is due to the fact that they provide a very convenient setting to
study evolution equations in a geometric setting \cite{BaerWafo,
  BaerGinoux, GerardBG, GerardWrochna, ElderingCR, ElderingThesis,
  Simonett, ShubinAsterisque, Thalmaier18}. For instance, maximal
regularity results on manifolds with bounded geometry were obtained by
H. Amann \cite{AmannParab, AmannCauchy} and A. Mazzucato and V. Nistor
\cite{MN1}. Applications to Quantum Field Theory were obtained
C. B\"ar and A. Strohmaier \cite{BaerStrohmaier}, by C. G\'erard
\cite{gerardRR}, and by W. Junker and E. Schrohe
\cite{SchrohePhysics}, to mention just a few of the many papers on the
subject. In this paper, however, we establish some basic results for
\emph{elliptic} equations, which are useful in the study of evolution
equation and for many other problems.

Let $M$ be a manifold with boundary and bounded geometry, see 
Definition~\ref{def_bdd_geo}, and consider
a boundary value problem of the form
\begin{align}\label{eq.mixed0}
  \left\{ \begin{aligned} P u & \ = f && \text{in } M\\
    C u & \ = h && \text{on } \pa M.
  \end{aligned} \right. 
\end{align}
More precisely, we are interested in \emph{regularity results}, which
generally say that if $u$ has a minimal regularity, but $f$ and $h$
have a good regularity, then $u$ also has a good regularity. See
Definition~\ref{def.loc.reg} for a precise definition.

We provide \emph{three main methods} to obtain a good regularity for
$u$ based on:
\begin{enumerate} 
\item the uniform regularity of the local problems;
\item the uniform Shapiro-Lopatinski conditions; and
\item the well-posedness in energy spaces.
\end{enumerate}

The first two methods provide a characterization of the boundary value
problems that satisfy regularity, but are not always easy to use in
practice. Most of the examples of boundary value problems that we know
that are satisfying regularity are obtained from a well-posedness
result in spaces of minimal regularity (energy spaces). In turn, these
well-posedness results are obtained from coercivity.

Let us discuss now each of the three methods and the specific results
and applications contained in this paper.

\subsubsection*{Uniform regularity of the local problems}
This method is discussed in Section~\ref{sec.u.est}. In that section,
we associate to each point $x$ of $M$ a local operator $P_x$. This
operator is associated to a boundary value problem $(P_x, C_x)$ if $x$
is on the boundary. These operators are defined on small--but uniform
size--coordinate patches around $x$ (some care is needed close to the
boundary). We introduce then \emph{uniform regularity conditions} for
these local operators. This amounts to regularity conditions for each
of the local operators in such a way that the resulting constants are
independent of $x$ (see Definition~\ref{def.unif.reg}).  We then show
that the initial operator satisfies regularity if, and only if, the
associated family of local operators satisfies a uniform regularity
condition.  This is Theorem~\ref{thm.reg0}. We further show that if
the family of local operators is compact and satisfies regularity,
then it satisfies a \emph{uniform} regularity condition.  The
compactness condition is satisfied, for instance, by strongly elliptic
operators endowed with \emph{Dirichlet boundary conditions}. This then
yields right away a regularity result for these operators (Theorem
\ref{thm.reg.D}).

\subsubsection*{Uniform Shapiro-Lopatinski conditions} 
This method is discussed in Section~\ref{sec.SL}.  In analogy with the
classical case, we introduce a \emph{uniform Shapiro-Lopatinski
  regularity condition} and we show that it characterizes the boundary
value problems that satisfy regularity (Theorem~\ref{thm.SL}). The
usual Shapiro-Lopatinski conditions are not expressed in a
quantitative way that will make them amenable to generalize right away
to uniform conditions. For this reason, we go back to a more basic
formulation of the Shapiro-Lopatinski condition in terms of the left
invertibility of suitable model problems $(P_x^{(0)}, C_x^{(0)})$, one
for each point of the boundary, and we require the norms of inverses
of these problems to be bounded. If the problems $(P_x, C_x)$ are the
\emph{local} versions of the initial problem $(P, C)$, the problems
$(P_x^{(0)}, C_x^{(0)})$ are the \emph{micro-local} versions of $(P,
C)$.

\subsubsection*{Well-posedness in energy spaces}
This method is discussed in Section~\ref{sec.coercive}. While
intellectually satisfying (especially in view of the classical
results), the uniform Shapiro-Lopatinski is not very easy to use in
practice. For this reason, we extend a result of Nirenberg
\cite{NirenbergSE} to prove that the problems that are well-posed in
energy spaces satisfy regularity (Theorem~\ref{thm.wp}). We then
combine this method with the other two to check that the Robin
boundary conditions $e_1\pa_{\nu}^{a} + Q$ (for suitable $Q$, see
Remark \ref{rem.Q20}) satisfy, first, the classical (i.e. not
necessarily uniform) Shapiro-Lopatinski boundary conditions, then that
they satisfy the compactness condition, and hence, the uniform
regularity for the local problems. This allows us to conclude a
regularity result for mixed Dirichlet/Robin boundary conditions (with
suitable $Q$). Let us explain this result in more detail.

For the rest of this introduction, $(M,g)$ will be a Riemannian
manifold with boundary and bounded geometry and $E \to M$ will be a
vector bundle with bounded geometry. We assume that we are given a
decomposition $E \vert_{\pa M} = F_0 \oplus F_1$ as an orthogonal sum
of vector bundles with bounded geometry and we let $e_j$ denote the
orthogonal projection $E \to F_j$. Let $a$ be a sesquilinear form on
$T^*M\otimes E$ and let $Q$ be a first order differential operator
acting on the sections of $F_1 \to \pa M$. To this data we associate
the operator $\tilde P_{(a, Q)} \colon H^1(M; E) \to H^1(M; E)^*$
defined by
\begin{equation}\label{eq.def.PaQ}
   \langle \tilde P_{(a, Q)}
     u, v \rangle  = \int_M a(\nabla u, \nabla v) d\vol_g + \int_{\pa
     M} (Q u, v) d\vol_{\pa g}.
\end{equation} 
We shall be interested in the second order operator $\tilde P \define
\tilde P_{(a, Q)} + Q_1$, where $Q_1$ is a first order differential
operator. An operator of this form will be called a \emph{second order
  differential operator in divergence form.} Let $\nu$ be the outer
normal vector field to the boundary and $r$ be the distance to the
boundary, then $dr(\nu)=1$ and $\pa_{\nu}^{a} u\define a(dr, \nabla
u)$ is the associated conormal derivative, see Remarks~\ref{rem.EndE},
\ref{rem.Q20} and Example~\ref{ex.pa.nu}.  We let $P$ be obtained from
$\tilde P$ using the restriction to $H^1_0$.  See
Section~\ref{subsub.bilin.forms} for more details.  Our result is a
regularity result for the problem
\begin{align}\label{eq.mixed}
  \left\{ \begin{aligned} P u &= f && \text{in } M\\
    e_0 u & = h_0 && \text{on } \pa M\\
    e_1 \pa_{\nu}^{a}u + Q u &= h_1 &&\text{on } \pa M \,.
  \end{aligned}\right.
\end{align}

Recall that a sesquilinear form $a$ on a hermitian vector bundle $V
\to X$ is \emph{called strongly} coercive (or \emph{strictly
  positive}) if there is some $c>0$ such that $\Re a(\xi,\xi)\geq
c|\xi|^2$ for all $x\in X$ and $\xi \in V_x$. If the sesquilinear form
$a$ on $T^*M \otimes E$ used to define $P$ is strongly coercive, then
$P$ is said to satisfy the \emph{strong Legendre condition}. See also
Definition~\ref{def.s.coercive}.

\begin{theorem}\label{thm.intro.reg}
Let $M$ be a manifold with boundary and bounded geometry, $E \to M$ be
a vector bundle with bounded geometry, and $\tilde P \define \tilde
P_{(a, Q)} + Q_1$ be second order differential operator in divergence
form acting on $(M, E)$, as above. Assume that $\tilde P$ has
coefficients in $W^{\ell+1, \infty}$, that $a$ that is strongly
coercive, and that $Q + Q^*$ is a zero order operator. Then there is
$C > 0$ such that, if $u \in H^1(M; E)$, $f \define Pu \in
H^{\ell-1}(M; E)$, $h_0 \define e_0u \vert_{\pa M} \in
H^{\ell+1/2}(\pa M; F_0)$, and $h_1 \define e_1 \pa_{\nu}^{a} u + Q u
\in H^{\ell-1/2}(\pa M; F_1)$, then $u \in H^{\ell+1}(M; E)$ and
\begin{equation*}
  \|u\|_{H^{\ell+1}} \leq C \left( \|f\|_{H^{\ell-1}} +
  \|h_0\|_{H^{\ell+1/2}} + \|h_1\|_{H^{\ell-1/2}} + \|u\|_{H^{1}}
  \right).
\end{equation*}
\end{theorem}

In Subsection~\ref{ssec.saw}, we will see why, in general, the bounded
geometry assumption is useful. Many of our results extend to higher
order equations, however, including this would have greatly extended
the length of the paper. See nevertheless Remarks
\ref{rem.higher.order} and \ref{rem.in.general}.

\subsubsection*{Contents of the paper}
The results in this paper are a natural continuation of our joint
paper with Bernd Ammann \cite{AGN1}. In that paper, we established
some general geometric results on manifolds with boundary and bounded
geometry and dealt almost exclusively with Dirichlet boundary
conditions. Also, in \cite{AGN1} we restricted ourselves to the case
of the Laplace operator.  The main contribution of this paper is the
fact that we consider general uniformly elliptic operators and we work
under general boundary conditions.  This requires some new, specific
ideas and results. Further results are included in \cite{AGN3}.

Here are the contents of the paper. Section~\ref{sec2} contains some
basic definitions and some background results. Several of these
results are from \cite{AGN1}. For instance, in this section we review
the needed facts on manifolds with bounded geometry and we recall the
definition of Sobolev spaces using partitions of unity. Section
\ref{sec3} contains some preliminary results on Sobolev spaces and
differential operators. In particular, we provide a description of
Sobolev on manifolds with bounded geometry spaces using vector
fields. The first two sections can be skipped at a first reading by
the seasoned researcher. Section~\ref{sec.variational} introduces our
differential operators and boundary value problems, both of which are
given in a variational (i.e. weak) form. We show that, under some mild
conditions, all non-degenerate boundary conditions are equivalent to
variational ones, which justifies us to concentrate on the later.
Also in this section we formulate the regularity condition. The last
three sections describe the three methods explained above and their
applications. In the last section, in addition to discussing the third
method mentioned above (``Well-posedness in energy spaces''), we also
introduce a \emph{uniform Agmon condition}, which, we show, is
equivalent to the coercivity of the operator. This then leads to an
application to the well-posedness of some parabolic equations on
manifolds with bounded geometry.

\paragraph{\textbf{Acknowledgments}}
We would like to thank Bernd Ammann for several useful discussions. We
also gladly acknowledge the hospitality of the SFB 1085 Higher
Invariants at the Faculty of Mathematics at the University of
Regensburg where parts of this article was written. We are also very
grateful to Alexander Engel and Mirela Kohr for carefully
reading our paper and for making several very useful comments.

\section{Background material and notation} \label{sec2}

We begin with some background material, for the benefit of the reader.
More precisely, we introduce here the needed function spaces and we
recall the needed results from \cite{AGN1} as well as from other
sources.  We also use this opportunity to fix the notation, which is
standard, so this section can be skipped at a first reading. This
section contains no new results. \emph{Throughout this paper, $M$ will
  be a (usually non-compact) connected smooth manifold, possibly with
  boundary.}

\subsection{General notations and definitions}\label{ssec.not} 
We begin with the most standard concepts and some notation.

\subsubsection{Continuous operators} 
Let $X$ and $Y$ be Banach spaces and $A \colon X \to Y$ be a linear
map. Recall that $A$ is \emph{continuous} (or \emph{bounded}) if, and
only if, $\|A\|_{{X,Y}} \define \inf_{x \neq 0} \frac{\|Ax\|}{\|x\|} <
\infty$. When the spaces on which $A$ acts are clear, we shall drop
them from the notation of the norm, thus write $\|A\| =
\|A\|_{{X,Y}}$. We say that $A$ is an \emph{isomorphism} if it is a
continuous bijection (in which case the inverse will also be
continuous, by the open mapping theorem).

\subsubsection{The conjugate dual spaces} 
For complex vector spaces $V$ and $W$, a \emph{sesquilinear} map
$V\times W\to \CC$ will always be (complex) linear in $V$ and
anti-linear in $W$. A \emph{Hermitian form} on $V$ is a positive
definite sesquilinear map $(\,\cdot\,,\,\cdot\,)\colon V\times V\to
\CC$, which implies $(w,v)=\overline{(v,w)}$. In order to deal with
the complex version of the Lax--Milgram lemma, we introduce the
following notation and conventions.

\begin{notation}\label{not.*}\normalfont{
Let $V$ be a complex vector space, usually a Hilbert space. We shall
denote by $\overline{V}$ the complex conjugate vector space to $V$.
If $V$ is endowed with a topology, we denote by $V'$ the (topological)
dual of $V$. It will be convenient to denote $V^* \define
(\overline{V})' \simeq \overline{V'} $. We can thus regard a
continuous sesquilinear form $ B \colon V \times W \to \CC$ as a
continuous bilinear form on $V \times \overline{W}$, or, moreover, as
a map $V \to W^*$. Let us assume that $V$ is a Hilbert space with
inner product $( \, , \, ) \colon V \times V \to \CC$.  If $T\colon V
\to W$ is a continuous map of \emph{Hilbert} spaces, then we denote by
$T^* \colon W^* \to V^*$ the \emph{adjoint} of $T$, as usual.}
\end{notation}

\subsubsection{Vector bundles} 
Let $E \to M$ be a smooth real or complex vector bundle endowed with
metric~$(\, ,\, )_E$. We denote by $\Gamma(M; E)$ the set of
\emph{smooth} sections of $E$. Let $E$ be endowed with a connection
\begin{equation*}
  \nabla^E \colon \Gamma(M; E) \to \Gamma(M; E \otimes T^{*}M).
\end{equation*}
We assume that $\nabla^E$ is metric preserving, which means that
\begin{equation*}
   X ( \xi, \eta )_E = ( \nabla_X \xi, \eta )_E + ( \xi, \nabla_X
   \eta )_E.
\end{equation*}
If $F \to M$ is another vector bundle with connection $\nabla^F$, then
we endow $E \otimes F$ with the induced product connection $\nabla^{E
  \otimes F}(\xi \otimes \eta) \define (\nabla^{E} \xi) \otimes \eta +
\xi \otimes \nabla^{F}\eta$.

We endow the tangent bundle $TM\to M$ with the Levi-Civita connection
$\nabla^M$, which is the unique torsion free, metric preserving
connection on $TM$.  We endow $T^*M$ and all the tensor product
bundles $E \otimes T^{*\otimes k}M \define E \otimes (T^*M)^{\otimes
  k}$ with the induced tensor product connections (we write
  $V^{\otimes k} := V \otimes V \otimes \ldots \otimes V$,
  $k$-times).

\subsection{Manifolds with boundary and bounded geometry}\label{sec.bg}
We now recall some basic material on manifolds with boundary and
bounded geometry, mostly from \cite{AGN1}, to which we refer for more
details and references.

\begin{definition}\label{def_ttly_bdd_curv} 
A vector bundle $E \to M$ with given connection is said to \emph{have
  totally bounded curvature} if its curvature and all its covariant
derivatives are bounded (that is, $\|\nabla^k R^E\|_\infty < \infty$
for all $k$). If $TM$ has totally bounded curvature, we shall then say
that $M$ has \emph{totally bounded curvature}.
\end{definition}

Let $\exp_p^M\colon T_pM \to M$ be the exponential map at $p$
associated to the metric and
\begin{align*}
\begin{gathered}
  \rinj(p) \define \sup\{ r \mid \exp_p^M \colon B_r^{T_pM}(0) \to M
  \text{ is a diffeomorphism onto its image} \}\\
 \rinj(M) \define \inf_{p \in M} \rinj(p).
\end{gathered}
\end{align*}

The following concept is classical and fundamental.

\begin{definition}\label{def_bdd_geo_no_bdy}
A Riemannian manifold without boundary $(M,g)$ is said to be of
\emph{bounded geometry} if $\rinj(M)>0$ and if $M$ has \emph{totally
  bounded curvature}.
\end{definition}

If $M$ has boundary, clearly $\rinj(M)=0$, so a manifold with non-empty
boundary will never have bounded geometry in the sense of the above
definition. However, Schick has found a way around this difficulty
\cite{Schick2001}. Let us recall a definition equivalent to his
following \cite{AGN1}. The main point is to describe the boundary as a
suitable submanifold of a manifold without boundary and with bounded
geometry.  Let us consider then a hypersurface $\subset M$, i.e.\ a
submanifold with $\dim H=\dim M -1$. Assume that $H$ carries a
globally defined unit normal vector field $\nu$ and let
$\exp^\perp(x,t) \define \exp^M_x(t\nu_x)$ be the exponential in the
direction of the chosen unit normal vector.  By $\II^N$ we denote the
\emph{second fundamental form} of $N$ (in $M$: $\II^N(X,Y)\nu \define
\nabla_XY - \nabla^{N}_XY$).

\begin{definition}\label{hyp_bdd_geo} 
Let $(M^{m},g)$ be a Riemannian manifold of bounded geometry with a
hypersurface $H = H^{m-1}\subset M$ and a unit normal field $\nu$ on
$H$. We say that $H$ is a \emph{bounded geometry hypersurface} in $M$
if the following conditions are fulfilled:
\begin{enumerate}[(i)]
\item $H$ is a closed subset of $M$;
\item $\|(\nabla^H)^k \II^H \|_{L^\infty} < \infty$ for all $k\geq 0$;
\item $\exp^\perp\colon H\times (-\delta,\delta)\to M$ is a
  diffeomorphism onto its image for some $\delta > 0$.
\end{enumerate}
We shall denote by $r_\pa$ the
largest value of $\delta$ satisfying this definition.
\end{definition}

We have shown in \cite{AGN1} that $(H, g|_H)$ is then a manifold of
bounded geometry in its own right. See also \cite{ElderingCR,
  ElderingThesis} for a larger class of submanifolds of manifolds with
bounded geometry.
  
\begin{definition}\label{def_bdd_geo}  
  A Riemannian manifold~$M$ with (smooth) boundary has \emph{bounded
    geometry} if there is a Riemannian manifold $\widehat M$ with
  bounded geometry satisfying
\begin{enumerate}[(i)]
\item $M$ is contained in $\widehat M$;
\item $\partial M$ is a bounded geometry hypersurface in $\widehat M$.
\end{enumerate}
\end{definition}

\begin{example} Lie manifolds have bounded geometry 
\cite{sobolev, aln1}. It follows that, Lie manifolds with
boundary are manifolds with boundary and bounded geometry.
\end{example}

For $x, y \in M$, we let $\dist(x, y)$ denote the distance between $x$
and $y$ with respect to the metric $g$. 

\subsection{Coverings, partitions of unity, and Sobolev spaces}
\label{ssec.cov}
We now introduce some basic notation and constructions for manifolds
with boundary and bounded geometry. The boundary $\pa M$ of a manifold
with boundary $M$ will always be a subset of $M$: i.e.  $\pa M \subset
M$. \emph{For the rest of the paper, we will always assume that $E\to
  M$ has totally bounded curvature, and, for the rest of this section,
  we shall assume that $M$ is a manifold with boundary and bounded
  geometry.}
  
Let $\nu$ be the inner unit normal vector field of $\partial M$ and
let $r_{\pa}$ as in the bounded geometry condition for $\pa M$ in $M$,
Definition~\ref{hyp_bdd_geo}.  Moreover, for a metric space $X$, we
shall denote by $B_r^X(p)$ the open ball of radius $r$ centered at $p$
and set $B_r^m(p)\ede B_r^{\mathbb R^m}(p)$. We shall identify $T_p M$
with $\RR^{m+1}$ and, respectively, $T_p \pa M$ with $\RR^{m}$, using
an orthonormal basis, thus obtaining a diffeomorphism
$\exp^{M}_p\colon B^{m+1}_{r}(0) \to B^{M}_r(p)$. For $r< \frac{1}{2}
\min\{ \rinj(\pa M), \rinj(M), r_\partial\}$ we define the maps
\begin{align*}
\left\{\begin{aligned} &\kappa_p \colon B^{m}_{2r}(0) \times [0,
  2r)\to M,&& \kappa_p(x, t) \define \exp^M_{q}(t\nu_{q}),&& \text{if
  } p \in \pa M,\ q \define \exp_{p}^{\pa M}(x)\\
  &\kappa_p \colon B^{m+1}_{r}(0) \to M,&& \kappa_p(v) \ede
 \exp_p^{M}(v),&& \text{if } \dist(p, \pa M) \geq r,
\end{aligned}\right.
\end{align*}
with range 
\begin{equation}\label{eq.Ugamma}
 U_p(r) \ede
 \begin{cases}
  \kappa_p(B^{m}_{2r}(0) \times [0,2r)) \subset M & \ \mbox{ if } p
    \in \pa M\\
  \kappa_p(B^{m+1}_{r}(0)) \ =\ \exp_p^{M}(B^{m+1}_{r}(0)) \subset M
  \ & \ \mbox{ otherwise.}
 \end{cases}
\end{equation}

\begin{definition}\label{FC-chart}
Let $r_{FC} \define \min\left\{\, \frac{1}{2} \rinj (\partial M), \,
\frac{1}{4} \rinj (M),\, \frac{1}{2} r_\partial\, \right\}$ and $0 < r
\le r_{FC}$. Then $\kappa_p \colon B^{m}_{r}(0) \times [0,r) \to
  U_p(r)$ is called a \emph{Fermi coordinate chart} at $p\in \partial
  M$.  The charts $\kappa_p$ for $\dist(p, \pa M) \geq r$ are called
  \emph{geodesic normal coordinates}.
\end{definition}

To define our Sobolev spaces, we need suitable coverings of our
manifold. For the sets in the covering that are away from the
boundary, we will use geodesic normal coordinates, whereas for the
sets that intersect the boundary, we will use the Fermi coordinates
introduced in Definition~\ref{FC-chart}.

\begin{definition}\label{FC}  
Let $M$ be a manifold with boundary and bounded geometry and let $0 <
r\leq r_{FC} \define \min\left\{\, \frac{1}{2} \rinj (\partial M), \,
\frac{1}{4} \rinj (M),\, \frac{1}{2}r_\partial\, \right\}$, as in
Definition~\ref{FC-chart}.  A subset $\{p_\gamma\}_{\gamma \in \NN}$
is called an \emph{$r$-covering subset of $M$} if the following
conditions are satisfied:
\begin{enumerate}[(i)]
\item For each $R>0$, there exists $N_R \in \NN$ such that, for each
  $p \in M$, the set $\{\gamma \in \NN\vert\, \dist(p_\gamma, p) <
  R\}$ has at most $N_R$ elements.
 \item For each $\gamma \in \NN$, we have either $p_\gamma \in \pa M$
   or $\dist (p_\gamma, \pa M) \geq r$, so that $U_\gamma \define
   U_{p_\gamma}(r)$ is defined, compare to \eqref{eq.Ugamma}.
 \item $M \subset \cup_{\gamma = 1}^{\infty} U_{\gamma}$.
\end{enumerate}
\end{definition}

\begin{remark}
If $0 < r < r_{FC}$, then we can always find an $r$-covering subset of
$M$, since $M$ is a manifold with boundary and bounded geometry
\cite[Remark~4.6]{GrosseSchneider}.  Moreover, it then follows
from (i) of Definition~\ref{FC} that the coverings $\{U_\gamma\}$ of
$M$ and $\{U_\gamma\cap \pa M\}$ of $\pa M$ are uniformly locally
finite.
\end{remark}

We shall need the following class of partitions of unity defined using
$r$-covering sets. Recall the definition of the sets $U_\gamma \ede
U_{p_\gamma}(r)$ from Equation~\eqref{eq.Ugamma}.

\begin{definition} \label{def_part}
A partition of unity $\{\phi_\gamma\}_{\gamma \in \NN}$ of $M$ is
called \emph{an $r$-uniform partition of unity associated to the
  $r$-covering set $\{p_\gamma\} \subset M$}, see Definition~\ref{FC},
if the support of each $\phi_\gamma$ is contained in $U_\gamma$ and
$\sup_{\gamma} \|\phi_\gamma\|_{W^{\ell,\infty}(M)} < \infty$ for each
fixed $\ell \in \ZZ_{+}$.
\end{definition}

In order to deal with boundary value problems with values in a vector
bundle (that is, with boundary value problems for systems), we will
also need the concept of synchronous trivializations, which we briefly
recall here:

\begin{definition}\label{def_sync}
Let $M$ be a Riemannian manifold with boundary and bounded geometry,
and let $E\to M$ be a Hermitian vector bundle with metric
connection. Let $(U_{\gamma} , \kappa_\gamma , \phi_\gamma)$ be Fermi
and geodesic normal coordinates on $M$ together with an associated
$r$-uniform partition of unity as in the definitions above.  If
$p_\gamma\in M\setminus U_r(\pa M)$, then $E|_{U_{\gamma}}$ is
trivialized by parallel transport along radial geodesics emanating
from $p_\gamma$. If $p_\gamma\in \pa M$, then we trivialize
$E|_{U_{\gamma}}$ as follows: First we trivialize $E|_{U_{\gamma}\cap
  \partial M}$ along the underlying geodesic normal coordinates on
$\pa M$.  Then, we trivialize by parallel transport along geodesics
emanating from $\pa M$ and being normal to $\pa M$.  The resulting
trivializations are called \emph{synchronous trivializations along
  Fermi coordinates} and are maps
\begin{align}\label{eq.synchronous}
  \xi_\gamma\colon \kappa_{\gamma}^{-1}(U_{\gamma})\times \CC^t \to
  E|_{U_{\gamma}}
  \end{align}
where $t$ is the rank of $E$.
\end{definition}

We shall need a definition of Sobolev spaces using partitions of unity
and ``Fermi coordinates'' \cite{GrosseSchneider} and a few
standard results. In the scalar case, these results were stated in
\cite{AGN1}. Here we stress the vector valued case. First, we have the
following proposition that is a direct consequence of Theorems 14 and
26 in \cite{GrosseSchneider}.

\begin{proposition}\label{prop.part.unit}
Let $M$ be a Riemannian manifold with boundary and bounded geometry.
Let $\{\phi_\gamma\}$ be an $r$-uniform partition of unity associated
to an $r$-covering set $\{p_\gamma\} \subset M$ and let $\kappa_\gamma
= \kappa_{p_\gamma}$ be as in Definition~\ref{FC-chart}.  Let $E \to
M$ be a vector bundle with totally bounded curvature with
trivializations $\xi_\gamma$ as in Definition~\ref{def_sync}. Then
  \begin{equation*}
  ||| u |||^{p} \define \sum_{\gamma} 
  \|\xi_\gamma^*(\phi_\gamma u)\|_{W^{s,p}}^{p}
 \end{equation*}
defines a norm equivalent to the standard norm on $W^{s,p}(M;E)$,
$s\in \RR$, $1<p<\infty$.
\end{proposition}

As in the scalar case \cite{AGN1}, the space $\Gamma_c(M; E)$ of
smooth, compactly supported sections of $E$ is dense in $W^{s,p}(M;
E)$, for $s\in \RR$ and $1<p<\infty$. This is obtained by truncating
the sum.  As usual, we shall let $H^\ell(M; E) \define W^{\ell, 2} (M;
E)$, $s \in \RR$.  Similarly, we shall need the following extension of
the trace theorem to the case of manifolds with boundary and bounded
geometry, see Theorem~27 in \cite{GrosseSchneider} (see
\cite{AGN1} for more references).

\begin{theorem}[Trace theorem]\label{thm.trace}
 Let $M$ be a manifold with boundary and bounded geometry and let $E
 \to M$ have totally bounded curvature. Then, for every $s > 1/2$, the
 restriction to the Dirichlet part of the boundary $\text{res}\colon
 \CIc(M) \to \CIc(\pa_D M)$ extends to a continuous, surjective map
 \begin{equation*}
  \operatorname{res}\colon  H^s(M; E) \, \to \, 
  H^{s-\frac{1}{2}}(\partial_D M; E).
 \end{equation*}
\end{theorem}

Let $\pa_{\nu}$ be the normal derivative at the boundary. We then
denote by $H^m_0(M; E)$ the kernel of the restrictions maps
$\operatorname{res}\circ \pa_{\nu}^j$, $0\leq j\leq m-1$. It is known
that $H^{-m}(M; E^*)$ identifies with $H^m_0(M; E)^*$.

See also \cite{Aubin, sobolev, Hebey1, HebeyBook, kordyukovLp1,
  kordyukovLp2, ShubinAsterisque, TriebelBG, TriebelGr} for related
results, in particular, for the use of the partitions of unity. See
\cite{EngelThesis} and its outgrowth \cite{EngelLongArt} for an
introduction to manifolds of bounded geometry. See
\cite{BrezisSobolev, JostBook, LionsMagenes1, Taylor1} for the general
results on Sobolev spaces not proved above.

\section{Preliminary results} \label{sec3}

We now include some preliminary results. This section consists mostly
new results, but advanced reader can nevertheless skip this section at
a first reading. \emph{In this section we assume that $M$ is a
  manifold with boundary and bounded geometry and that $E \to M$ is a
  vector bundle with totally bounded curvature (and, hence, with
  bounded geometry).}

\subsection{Alternative characterizations of Sobolev spaces}
\label{ssec.alt}
We now provide an alternative description of Sobolev spaces using
vector fields. If $A \colon V_1 \to V_2$ is a linear map of normed
spaces, we define its \emph{minimum reduced modulus} $\gamma(A)$ by
\begin{equation*}
  \gamma(A) \define \inf_{\xi\in V_1 / \ker (A)} \ \sup_{\eta
    \in \ker (A)}\, \frac{\|A\xi\|}{\|\xi+\eta\|}\, ,
\end{equation*}
i.e.\ $\gamma(A)$ is the largest number $\gamma$ satisfying $\|A \xi\|
\ge \gamma \|\xi + \ker(A)\|_{V_1/\ker (A)}$ for all $\xi\in V_1$
where $ \|\,\cdot\,\|_{V_1/\ker (A)}$ is the quotient norm on
$V_1/\ker (A)$. As it is well known, if $V_1$ is complete, a standard
application of the Open Mapping Theorem gives that the image of $A$ is
closed if, and only if, $\gamma(A) > 0$. If $V_1$ and $V_2$ are
Hilbert spaces and $A$ is surjective, we have $\gamma(A)^{-2} =
\|(AA^*)^{-1}\|$.

We have the following alternative characterization of Sobolev spaces
in terms of vector fields. This alternative definition is more
intuitive and easier to use in analysis. It is based on a the choice
of a suitable finite family of vector fields, whose existence is
assured by the following lemma.

\begin{lemma} \label{lemma.prop.exists}
Let $(M,g)$ be a manifold with boundary and of bounded geometry. Then
there exist vector fields $X_1, X_2, \ldots, X_N \in W^{\infty,
  \infty}(M; TM)$ such that for any $x\in M$ the map $\Phi_x\colon
\RR^N \ni (\lambda_1,\ldots,\lambda_N) \mapsto \sum \lambda_i X_i(x)
\in T_xM$ is onto at any $x$ and $\inf_x \gamma(\Phi_x) > 0$. Then we
have that $[X_i, X_j] = \sum_{k=1}^N C_{ij}^k X_k$ and
$\nabla_{X_i}X_j = \sum_{k=1}^N G_{ij}^k X_k$ for some functions
$G_{ij}^k, C_{ij}^k \in W^{\infty, \infty}$ and any $1 \leq i, j \leq
N$.  Moreover, we can choose $X_1$ of length one and normal to the
boundary of $M$ and $X_j$ tangent to the boundary for $j > 1$.
\end{lemma}

\begin{proof}
We use a covering by Fermi coordinates resp. geodesic normal
coordinates $(U_\beta, \kappa_\beta)$ as in Definition~\ref{FC}
together with an associated partition of unity $\phi_\beta$ as in
Definition~\ref{def_part}. We recall that all transition function
$\kappa_\alpha\circ \kappa_\beta^{-1}$, the $\phi_\beta$ and all their
derivatives are uniformly bounded, see \cite{GrosseSchneider,
  ShubinAsterisque}. Let $x_\beta^i$ be the coordinates corresponding
to the chart $\kappa_\beta$. Note that the Christoffel symbols in each
chart and their derivatives are also uniformly bounded due to the
bounded geometry, \cite[Lemma 3.10]{GrosseSchneider}.

Let those coordinates be ordered such that $x_\beta^1=r$, the distance
to the boundary, for $U_\beta\cap \partial M\neq \varnothing$. Let
$X_1\define \sum_{ U_\beta\cap \partial M\neq \varnothing} \phi_\beta
\partial_r$. Then, $X_1$ is normal to $\partial M$. 

To define $X_j$ for $j>1$ we divide $\{U_\beta\}_\beta$ into
finitely many disjoints subsets $V_i$, $i=1,\ldots, d$, such that the
$U_\beta$'s in any $V_i$ are pairwise disjoint. This is always possible 
since the covering $\{U_\beta\}_\beta$ is uniformly locally finite. We set
\begin{align*}
X_{i,j}&\define \sum_{U_\beta\in V_i} \phi_\beta \partial_{x_\beta^j} 
\quad \text{for }j\in \{2,\ldots, m\}\\
X_{i,1}&\define \sum_{U_\beta\in V_i, U_\beta\cap \partial M=\varnothing} 
\phi_\beta \partial_{x_\beta^1}
\end{align*}
Then, $X_1,X_{i,j}\in W^{\infty,\infty}(M,TM)$. Moreover, all
$X_{i,j}$ are tangent to the boundary by construction.  By renaming
the vector fields we obtain $X_1, \ldots, X_N$.  Then $\Phi_x$ is onto
by construction for all $x\in M$.

Next we show that the reduced minimum modulus is uniformly bounded
from below: Let $x\in M$. Since the covering is uniformly locally
finite, there is a chart $\kappa_\alpha$ around $x$ with
$\phi_\alpha(x)>c>0$, $c$ independent of $x$. Since
$\partial_{x_\alpha}^i$ span $T_xM$ for all $(\lambda_1,\ldots,
\lambda_N)\in \RR^N$ there is a $(\mu_1,\ldots, \mu_N)\in \RR^N$ with
$\mu_j=0$ if $X_j$ was not build from
$\kappa_\beta$. Thus, \[\gamma(\Phi_x)\geq \inf_{(\lambda_1,\ldots,
  \lambda_N)\in \RR^N} \frac{| \sum_i \mu_i X_i(x)|}{\sum_i |\mu_i|^2}
\geq \inf_{(\lambda_1,\ldots, \lambda_N)\in \RR^N} \frac{c \sum_i
  \mu_i^2 }{\sum_i \mu_i^2} =c.\]

We have $\partial_{x_\beta^j} =\sum_i a_j^i\partial_{x_\alpha^i}$ with
$a_j^i\in W^{\infty,\infty}$ with bounds independent of $\alpha$ and
$\beta$ since this is true for the transition functions
$\kappa_\alpha\circ \kappa_\beta^{-1}$. Thus, together with the
uniform bounds on the corresponding Christoffel symbols we get
$\nabla_{\partial_{x_\alpha^i}}\partial_{x_\beta^j} = \sum_{k}
b_{ij}^k \partial_{x_\alpha^k}$ with $b_{ij}^k\in
W^{\infty,\infty}$. Altogether, we have

\begin{align*}
  \nabla_{X_i}X_j & =\sum_{U_\beta\in V_i, U_\alpha\in V_j} \phi_\beta
  \nabla_{\partial_{x_\beta^i}} (\phi_\alpha
  \partial_{x_\alpha^j})=\sum_{U_\beta\in V_i, U_\alpha\in V_j,k}
  \phi_\beta (\delta_k^j \partial_{x_\beta^i} \phi_\alpha +
  \phi_\alpha b_{ij}^k) \partial_{x_\beta^k}\\
  & =\sum_{k} G_{ij}^k X_k\quad \text{with }G_{ij}^k=
  \sum_{U_\beta\in V_i, U_\alpha\in V_j}(\delta_k^j
  \partial_{x_\beta^i} \phi_\alpha + \phi_\alpha b_{ij}^k)\in
  W^{\infty,\infty}
\end{align*}
where we used in the last step that the supports of any two $U_\beta$
in $V_i$ are disjoint.

By, $[X_i,X_j]=\nabla_{X_i}X_j-\nabla_{X_j}X_i$ the remaining claim
follows.
\end{proof}

\begin{proposition}\label{prop.alternative}
Let $X_1, X_2, \ldots, X_N \in W^{\infty, \infty}(M; TM)$ be such that
for any $x\in M$ the map $\Phi_x\colon \RR^N \ni
(\lambda_1,\ldots,\lambda_N) \mapsto \sum \lambda_i X_i(x) \in T_xM$
is onto at any $x$ and $\inf_x \gamma(\Phi_x) > 0$. Then
\begin{equation*}
 W^{\ell, p} (M; E) = \{\, u \, \vert \ \nabla^E_{X_{k_1}}
 \nabla^E_{X_{k_2}} \ldots \nabla^E_{X_{k_j}} u \in L^p(M; E),\ j \leq
 \ell, \ k_1,\ldots,k_j \leq N \, \} ,
\end{equation*}
$1 \leq p \leq \infty$.
If $E$ has furthermore totally bounded curvature, we have also
\begin{equation*}
 W^{\ell, p} (M; E) = \{\, u \, \vert \ \nabla^E_{X_{k_1}}
 \nabla^E_{X_{k_2}} \ldots \nabla^E_{X_{k_j}} u \in L^p(M; E) \, \} ,
\end{equation*}
where, this time, $1 \leq k_1 \leq k_2\le \ldots \leq k_j \leq N$, $j \le
\ell$.
\end{proposition}

We notice that the condition that $\Phi_x \colon \RR^N \to T_x M$ be
surjective implies that $\dim \ker \Phi_x=N-\dim M$ is constant.  For
compact $M$ then $\inf_{x\in M}\gamma(\Phi_x) > 0$ is automatically
satisfied.
\begin{proof} Let 
\begin{equation*}
 \maX \define \{\, u \, \vert \ \nabla^E_{X_{k_1}} \nabla^E_{X_{k_2}}
 \ldots \nabla^E_{X_{k_j}} u \in L^p(M; E),\ j \leq \ell,
 \ k_1,\ldots,k_j \leq N \}.
\end{equation*}
The fact that the vector fields $X_k$ are bounded with bounded
covariant derivatives gives that $\nabla_{X_k}\colon W^{j+1, p} (M; E)
\to W^{j, p} (M; E)$ is bounded for any $j \geq 0$ and any~$p$. This
gives $W^{\ell, p} (M; E) \subset \maX$.

To prove the converse, we proceed by induction as follows. First of
all, the assumption that $\inf_x \gamma(\Phi_x) > 0$ together with the
fact that $\Phi_x$ is surjective for all~$x$ gives that
$\Phi_x^*(\Phi_x \Phi_x^*)^{-1}$ is a (pointwise bounded) right
inverse to $\Phi_x$, regarded as a map $(\underline{\RR}^N)^{\otimes
  \ell} \to TM^{\otimes \ell}$. (Here $\underline{\RR}^N$ is the {\em
  trivial} vector bundle of rank $N$.) Let us denote by $\Psi$ the
extension of this map to sections of the corresponding vector bundles,
and, by abuse of notation, also by
\begin{equation*}
 \Psi \colon L^p(M; T^*M^{\otimes \ell} \otimes E) \to L^p(M;
 (\underline{\RR}^N)^{\otimes \ell} \otimes E)
\end{equation*}
the induced maps. Then $\Psi$ is continuous and a right inverse of
(the map defined by) $\Phi$. Therefore $\Psi$ is a homeomorphism onto
its image.  That gives that $\Psi(\xi) \in L^p(M;
(\underline{\RR}^N)^{\otimes \ell} \otimes E)$ if, and only if, $\xi
\in L^p(M; T^*M^{\otimes \ell} \otimes E)$. By taking the $(k_1, k_2,
\ldots, k_{\ell})$ component of $\Psi(\xi)$, we obtain that
$(\nabla^E)^\ell u \in L^p(M; T^*M^{\otimes \ell} \otimes E)$ if, and
only if, $\<(\nabla^E)^\ell u , X_{k_1} \otimes X_{k_2} \otimes \ldots
\otimes X_{k_\ell}\> \in L^p(M; E)$ for all $k_1, k_2, \ldots, k_\ell
\in \{ 1, 2, \ldots , N\}$.  To use induction, we notice that $\<
(\nabla^E)^\ell u , X_{k_1} \otimes X_{k_2} \otimes \ldots \otimes
X_{k_\ell}\> - \nabla^E_{X_{k_1}} \nabla^E_{X_{k_2}} \ldots
\nabla^E_{X_{k_\ell}} u $ is given by lower order terms of the same
kind.

For the last part, we also notice that we can commute $\nabla^E_{X_i}$
with $\nabla^E_{X_j}$ up to totally bounded terms using Lemma
\ref{lemma.prop.exists}.
\end{proof}

\subsection{Sobolev spaces without using connections}
Recall that we assume that $E\to M$ has totally bounded
  curvature. One way to think of such vector bundles is given by the
following lemma:

\begin{lemma}\label{lemma.embedding} Let us assume that $M$ is a manifold
with bounded geometry (possibly with boundary) and that $E \to M$ is a
Hermitian vector bundle of totally bounded curvature. Then there
exists a fibrewise isometric embedding $E \subset M \times \CC^N$ into
the trivial $N$-dimensional vector bundle with the standard metric
such that, if $e$ denotes the orthogonal projection onto $E$, then $e
\in M_N(W^{\infty, \infty}(M))$ and the connection of $E$ is
equivalent to the Grassmann (projection) connection of the embedding
(i.e. the difference of both connections is in
$W^{\infty,\infty}$). Conversely, if $e \in M_N(W^{\infty,
  \infty}(M))$ and $E \define e (M \times \CC^N)$, then $E$ with the
Grassmann connection has totally bounded curvature.
\end{lemma}

\begin{proof}
Let us consider for each open subset $U_\gamma$ as above the
synchronous trivialization $\xi_\gamma \colon E \vert_{U_\gamma} \to
\kappa_{\gamma}^{-1}(U_{\gamma}) \times \CC^t$ from
Equation~\eqref{eq.synchronous}, with $\CC^t$ the typical fiber above
$p_\gamma$. Then $\phi_\gamma^{1/2} \xi_\gamma$ extends to a vector
bundle map $E \to M \times \CC^t$ that is in $W^{\infty,\infty}$ since
$M$ has bounded geometry and the connection on $E$ has totally bounded
curvature. Let $N = N_{5r}$, with $N_{5r}$ as in Definition~\ref{FC}.
By the construction of the sets $U_\gamma$, we can divide the set of
all $\gamma$'s into $N+1$ disjoint subsets $\Gamma_k$, such that, for
each fixed $k$ and any $\gamma, \gamma' \in \Gamma_k$, the sets
$U_\gamma$ and $U_{\gamma'}$ are disjoint, by the construction of the
sets $U_\gamma$.  Let $\Psi_k \define \sum_{\gamma \in \Gamma_k}
\phi_\gamma^{1/2} \xi_\gamma$ and $\Psi \define (\Psi_1, \Psi_2,
\ldots, \Psi_{N+1}) \colon E \to M \times \CC^{t(N+1)}$ be the
resulting bundle morphism. Then $\Psi$ is isometric, it is in
$W^{\infty, \infty}$, and hence $e \define \Psi \Psi^{*}$ is the
desired projection.  The equivalence of the connections follows from
the uniform boundedness of the Christoffel symbols and their
derivatives associated to the synchronous trivialization \cite[Remark
  5.3]{GrosseSchneider}.
\end{proof}

This lemma may be used to reduce differential operators acting on
vector bundles to matrices of scalar differential operators.  It also
gives the following characterization of Sobolev spaces of sections of
a vector bundle.

\begin{proposition} \label{prop.projection}
We use the notation of the last lemma. Then
$$W^{k, p}(M; E) = e W^{k, p}(M)^N, \quad 1 \leq p \leq \infty.$$
\end{proposition}

We can now use this proposition to derive a description of Sobolev
spaces on manifolds with bounded geometry that is completely
independent of the use of connections.

\begin{remark}\label{rem.no.nabla}
The standard definition of the norm on Sobolev spaces is using powers
of $\nabla$ \cite{AGN1, GrosseSchneider}.  For instance $W^{k,
  \infty}(M) \define \{ u \, \vert \ \nabla^j u \in L^\infty(M), 0
\leq j \leq k \}$ (alternatively, it is the space of functions with
uniformly bounded derivatives of order $\le k$ in any normal geodesic
coordinate chart on $B_r^m$, for any fixed $r$ less than the
injectivity radius $\rinj(M)$ of $M$). We can define $W^{k, \infty}(M;
TM)$ similarly. Let then $X_1, X_2, \ldots, X_N$ as in Proposition
\ref{prop.alternative} and Lemma \ref{lemma.prop.exists}, $1 \leq p <
\infty$. Then
\begin{equation*}
 W^{\ell, p} (M) = \{\, u \, \vert \ X_{k_1} X_{k_2} \ldots X_{k_j} u
 \in L^p(M),\ j \leq \ell, \ 1 \leq k_1 \leq \ldots \leq k_j \leq N \,
 \}.
\end{equation*}
Together with Proposition \ref{prop.projection}, this gives a
description of Sobolev spaces without using connections.
\end{remark}

\subsection{Differential operators and partitions of unity}\label{sec:gd}
A \emph{differential operator on $E$} is an expression of the form $Pu
\ede \sum_{j=0}^k a_j \nabla^ju,$ with $a_j$ a section of $\End(E)
\otimes TM^{\otimes j}$. It can thus simply be regarded as a formal
collection of coefficients.  In particular, we do not identify the
differential operator with the maps that it induces (since it induces
many). A differential operator $Pu = \sum_{j=0}^k a_j \nabla^ju $ will
be said to \emph{have coefficients in} $W^{\ell, \infty}$ if $a_j \in
W^{\ell, \infty}(M; \End(E) \otimes TM^{\otimes j})$.  If $\ell = 0$,
we shall say that $P$ has \emph{bounded} coefficients.  If $\ell =
\infty$, we shall say that $P$ has \emph{totally bounded}
coefficients. The continuity of the contraction map
\begin{align*}
 W^{\ell, \infty}(M; \End(E) \otimes TM^{\otimes j}) \otimes W^{\ell,
   p}(M; T^*M^{\otimes j}\otimes E) \, \to \, W^{\ell, p}(M; E),
\end{align*}
gives that a differential operator $P = \sum_{j=0}^k a_j \nabla^j$
with coefficients in $W^{\ell, \infty}$ defines a \emph{continuous
  map}
\begin{align*}
 P \seq \sum_{j=0}^k a_j \nabla^j \colon W^{\ell+k, p}(M; E) \, \to \,
 W^{\ell, p}(M; E), \quad \ell \geq 0.
\end{align*}

\begin{lemma}\label{lemma.commutator}
Let $k\geq 0$ and let $P$ be an order $\ell$ differential operator
with coefficients $a_j\in W^{k+1, \infty}(M; \End(E)\otimes
TM^{\otimes j})$.  Let $\phi\in W^{k + \ell + 1, \infty}(M)$. Then the
commutator $[P, \phi]$ defines a continuous linear map $H^{k+\ell}(M;
E) \to H^{k+1}(M; E)$.  Moreover, if $\{\phi_\gamma\}_\gamma$ is a
bounded family in $W^{k + \ell + 1, \infty}(M)$, then operator norms
of $[P,\phi_\gamma]\colon H^{k+\ell}(M; E) \to H^{k+1}(M; E)$ are
bounded uniformly in $\gamma$.
\end{lemma}

\begin{proof}
We have $[P,\phi]u= \sum_{j=1}^{\ell} a_j \sum_{s=0}^{j-1}
\binom{j}{s} \nabla^{j-s}\phi \nabla^s u$.  Thus,
\begin{align*}
  \Vert [P,\phi]u\Vert_{H^{k+1}} \leq & \, C \sum_{r=0}^{k+1}
  \sum_{j=0}^\ell \Vert \nabla^r (a_j \sum_{s=0}^{j-1}
  \nabla^{j-s}\phi \nabla^s u)\Vert_{L^2}\\
\leq & \, C \sum_{r=0}^{k+1} \sum_{t=0}^{r}\sum_{j=0}^\ell \Vert
\nabla^{r-t} a_j \nabla^t(\sum_{s=0}^{j-1} \nabla^{j-s}\phi \nabla^s
u)\Vert_{L^2}
\end{align*}
and the claim follows by the regularity assumptions on $a_j$ and
$\phi$.
\end{proof}

\section{Variational boundary conditions and regularity}
\label{sec.variational}

We now introduce ``differential operators in divergence form'' from a
global point of view. The natural boundary value problem associated to
differential operators in divergence form will be called
\emph{variational boundary value problems.} In this subsection, we
introduce and take a first look at these variational boundary value
problems. We will see that, under some mild assumptions on our
differential operator $P$, any non-degenerate boundary value problem
is equivalent to a variational one. This allows to reduce the study of
the former to that of the latter, for which several general regularity
results will be obtained in the following sections. On the other hand,
the degenerate boundary value problems are known to behave in a
significantly different way than the non-degenerate ones (see, for
instance, \cite{NazarovPopoff} and the references therein).  The
possibility of reducing non-degenerate boundary value problems to
variational ones seems not to have been explored too much in the
literature.

We will continue to assume that $M$ is a smooth manifold with smooth
boundary. Moreover, we will assume that $TM$ has totally bounded
curvature, but \emph{we will not assume that $M$ has bounded
  geometry,} since we want to allow $M$ to be a domain in a Euclidean
space.  Recall that all differential operators are assumed to have
bounded coefficients and $E \to M$ has totally bounded
  curvature.

\subsection{Sesquilinear forms and operators in divergence form}
\label{subsub.bilin.forms}
It is important in applications to consider operators ``in divergence
form,'' which we will define below shortly. They provide a slightly
different class of differential operators than the operators with
coefficients in $L^\infty$ considered above and will be useful in
order to treat the Dirichlet and Robin problems on the same
footing. To introduce second order differential operators in
divergence form, we shall need the following data and assumptions:

\begin{assumption} \label{assume}
\normalfont{Let $M$ be a smooth manifold with smooth boundary, $E \to
  M$ a vector bundle, $a$ a sesquilinear form on $T^*M \otimes E$, and
  first order differential operators $Q$ and $Q_1$ satisfying
\begin{enumerate}
\item[\bf (A1)] $TM$ and $E$ have totally bounded curvature.

\item[\bf (A2)] $E\vert_{\pa M} = F_0 \oplus F_1$, with $F_0$ and
  $F_1$ with totally bounded curvature and $F_0$ is the orthogonal
  complement of $F_1$.
  
\item[\bf (A3)] $a = (a_x)_{x \in M}$ is a measurable, bounded family
  of hermitian sesquilinear forms
\begin{equation*}
  a = (a_x)_{x \in M} , \quad a_x\colon T^*_xM \otimes E_x \times
  T^*_xM \otimes E_x\to \CC.
\end{equation*}

\item[\bf (A4)] $Q_1$ has $L^\infty$ coefficients and acts on $(M,
  E)$.

\item[\bf (A5)] $Q$ has $W^{1, \infty}$ coefficients and acts on $(\pa
  M, F_1)$.

\item[\bf (A6)] $V$ is the closed subspace $H^1_0(M; E) \subset
  V\subset H^1(M; E)$ defined by $$V := \{ u \in H^1(M; E)\,
  \vert\ u\vert_{\pa M} \in \Gamma(\pa M; F_1)\}.$$
\end{enumerate}}
\end{assumption}

The family $(a_x)$ defines a section $a$ of the bundle $((T^*M\otimes
E) \otimes (T^{*}M\otimes \overline E))'$. In general, we say that a
section $a=(a_x)_{x\in M}$ is a \emph{bounded sesquilinear form on
  $T^*M \otimes E$} if it is an $L^\infty$-section of $((T^*M\otimes
E) \otimes (T^{*}M\otimes \overline E))'$.

\subsubsection{The Dirichlet (sesquilinear) form}
Using our assumptions \ref{assume}, we first define
\begin{align}\label{eq.def.Ba}
 B_{a}(u, v) \define \int_{M} a(\nabla u, \nabla v)\dvol_{g},
\end{align}
which is the \emph{Dirichlet form} associated to $a = (a_x)_{x \in
  M}$. Then the \emph{Dirichlet form} $B \colon V \times V \to \CC $
associated to $a$, $Q$, and $Q_1$ is
\begin{align}\label{eq.def.B}
  B(u, v) \define B_{a}(u, v) + (Q_1u, v)_{L^2(M; E)} + (Q u\vert_{\pa
    M}, v\vert_{\pa M})_{L^2(\pa M; F_1)},
\end{align}
where, initially, $u, v \in V \cap H^\infty(M; E)$, and then we then
extend $B$ to a sesquilinear linear form $B \colon V \times V \to \CC$
by continuity. In the future, we shall usually write $(Q u,
  v)_{L^2(\pa M)}$ instead of $(Q u\vert_{\pa M}, v\vert_{\pa
    M})_{L^2(\pa M; F_1)}$.

\subsubsection{The induced operator $\tilde P$}
The continuous, sesquilinear form $B \colon V \times V \to \CC$
defines a linear map (or operator)
\begin{align}\label{eq.def.tildeP}
\begin{gathered}
 \tilde P \colon V \, \longrightarrow \, V^* \\
 \< \tilde P (v), w \> \define B(v,w) \,, \quad v, w \in V\,.
\end{gathered}
\end{align}
If $B = B_{a}$, we shall denote by $\tilde P_{a}$ the associated
operator.  We note that $B$ and $\tilde P$ depend on the choice of the
metric $g$, although this will typically \emph{not} be shown in the
notation, since the metric will be fixed.

\begin{definition}\label{def.divergence}
We shall say that a differential operator $\tilde P \colon V \to V^*$
obtained as in Equation \eqref{eq.def.tildeP} is a \emph{second order
  differential operator in divergence form.}
\end{definition}

Recall that $H^{-1}(M; E^*) \simeq H^1_0(M; E)^*$. Using the metric on
$E$, we shall identify $H^{-1}(M; E^*) \simeq H^{-1}(M; E)$. Since
$H^1_0(M; E) \subset V$, we obtain the natural map $V^* \to H^1_0(M;
E)^* \simeq H^{-1}(M; E^*)$, and hence $\tilde P$ gives rise to a map
$P \colon V \to H^{-1}(M; E^*)$. Clearly, $P$ does not depend on $Q$,
whereas this is in general not the case for $\tilde P$. In fact, $Q$
only enters in the boundary conditions, see Example
\ref{ex.pa.nu}. Similarly, $\tilde P$ extends to a map $\tilde P :
H^1(M; E) \to H^1(M; E)^*$. We notice that $P$ (and hence also $\tilde
P$) determines the form $a$, which is the principal symbol of $P$. We
shall say that $a$ is the sesquilinear form associated to $P$. More
precisely, let us identify $((T^*M\otimes E) \otimes (T^{*}M\otimes
\overline E))' \simeq (T^*M \otimes T^{*}M)' \otimes \End(E)$ using
the metric on $E$. Then the quadratic function
\begin{equation}\label{eq.princ.symb}
   T^*M \ni \xi \to a(\xi, \xi) \in \End(E)
\end{equation}
is the principal symbol of $P$. In fact, it would be actually more
natural to start with $a \in TM \otimes TM \otimes \End(E)$. However,
as we \emph{always} consider a metric on $E$, this makes not difference
for us.

In a certain way, $P$ is the ``true'' differential operator associated
to $B$, whereas $\tilde P$ includes, in addition to $P$, also boundary
terms. To better understand this statement as well as the difference
between $P$ and $\tilde P$, a calculation in local coordinates is
contained in Example \ref{ex.pa.nu} below, see also
\cite{LionsMagenes1, Daners2000}.

\subsection{Variational boundary value problems}
We now examine the relation between the operators in divergence form
(i.e. of the form $\tilde P$) and boundary value problems. In
particular, we discuss the weak formulation of the Robin problem. See
also \cite{Agranovich07, KohrWendland18, LionsMagenes1, McLeanBook,
  OleinikBook92, Paltsev96} for the weak formulation of boundary value
problems. We assume that we are given a decomposition
\begin{equation}\label{eq.decomp}
  E\vert_{\pa M} = F_0 \oplus F_1 
\end{equation}
of the restriction of $E$ to the boundary into a direct sum of two
vector bundles \emph{with totally bounded curvature.} We consider
boundary differentials operators
\begin{equation}\label{eq.boundary.ops}
 C_j \colon H^\infty(M; E) \to L^2(\pa M; F_j),
\end{equation}
where $C_0$ is of order zero and $C_1 = C_{10} + C_{11}\pa_{\nu}$,
with $C_{11}$ and $C_{00}$ of order zero, $\pa_{\nu}$ the covariant
normal derivative at the boundary in the direction of the outer unit
normal vector $\nu$, and $C_{10}$ only including derivatives
tangential to $\pa M$. Each of the operators $C_0$, $C_{10}$, and
$C_{11}$ factors through a map $H^\infty(\pa M; E_{\vert \pa M}) \to
L^2(\pa M; F_j)$, which will be denoted with the same symbol and we
will assume to be differential operators with bounded coefficients. If
$C_0^{-1} \in L^\infty(\pa M; \End(F_0))$ and $C_{11}^{-1} \in
L^\infty(\pa M; \End(F_1))$ we shall say that the boundary conditions
$C = (C_0, C_1)$ are \emph{non-degenerate}.

We are interested in boundary value problems of the form
\eqref{eq.mixed0} where $h = (h_0, h_1)$, $h_j \in \Gamma(\pa M;
F_j)$, and the relation $Cu =h$ means $C_0 u = h_0$ and $C_1 u =
h_1$. We shall regard the boundary conditions $(C_0, C_1)$ and $(C_0',
C_1')$ as \emph{equivalent} if $(C_0', C_1') = (D_0 C_0, D_1 C_1)$,
where $D_j$, for each $j$, is a bounded automorphism of $F_j$ with
bounded inverse.  The two sets of solutions of two equivalent boundary
value problems are in a natural (continuous) bijection, which
justifies looking at equivalence classes of boundary value
problems. More precisely, let $C$ and $C'$ be equivalent boundary
value problems and $(D_0, D_1)$ be the automorphisms implementing the
equivalence.  Assume $C$ and $C'$ have totally bounded coefficients,
for simplicity, and let $h_0 \in H^{k+1/2}(\pa M; F_0)$, $h_1 \in
H^{k-1/2}(\pa M; F_1)$, $h \define (h_0, h_1)$ and $h':= (D_0 h_0, D_1
h_1)$.  Then the boundary value problems $Pu = f$, $Cu = h$ and $Pu =
f$, $C' u = h'$ have the same solutions $u \in H^{k+1}(M; E)$.

\subsubsection{Definition of variational boundary conditions}
Let $e_j$ be the orthogonal projection onto $F_j$ at the
boundary. Also, let $B$ be the basic sesquilinear form defined in
Equation \eqref{eq.def.B} of the previous subsection. Thus
consider \emph{for the rest of this paper}
\begin{equation}\label{eq.def.VnW}
  V \define \{ u \in H^1(M; E)\, | \ e_0u = 0 \mbox{ on } \pa M \}.
\end{equation} 
Spaces $V$ of this form with $e_0$ non-trivial (i.e. different from
the restriction to some components of the boundary) arises in the
study of the Hodge-Laplacian and in the study of the Ricci flow
\cite{Pulemotov}. Let $k\geq 1$. For any $h \in H^{k-1/2}(\pa M; F_1)$
and $f \in H^{k-1}(M; E)$, we let
\begin{equation*}
   F(v) \define \int_{M} (f, v) \dvol_g + \int_{\pa M} (h, v)\,
   d\vol_{\pa g}\ ,
\end{equation*} 
where $d\vol_{\pa g}$ is the induced volume form on $\pa M$. Then $F$
defines a linear functional on $\overline{V}$ and hence an element of
$V^*$. We denote $j_{k-1}(f, h) \define F$ the induced map
\begin{equation}\label{eq.def.jk}
  j_{k-1} \colon  H^{k -1}(M; E) \oplus H^{k
    - 1/2}(\pa M; F_1) \to V^*.
\end{equation}
The formula for $F$ makes sense also for $k = 0$ and $f=0$, in which
case it is just the dual map to the restriction at the boundary.
Recall that all our differential operators, including $\tilde P$, have
bounded coefficients.

\begin{remark}\label{rem.EndE}
We identify $(T^*M \otimes E \otimes T^*M \otimes E)'$ with ${(T^*M
  \otimes T^*M)'} \otimes \End(E)$ using the metrics on $TM$ and $E$,
respectively, as above. Then $a(dr, dr)$ can be regarded as a section
of $\End(E)\vert_{\pa M}$ and we have on $V$ that $a(dr, \nabla .) =
a(dr, dr) \pa_{\nu} + Q'$, where $Q'$ is a differential operator on
$\Gamma(\pa M; E|_{\pa M})$ (that is, it does not involve normal
derivatives at the boundary; it involves only tangential derivatives).
\end{remark}

Note that we are not assuming the bundles $F_i$ to be orthogonal.

\begin{remark} \label{rem.Q20}
Let us assume, in the definition of $B(u,v)$, Equation
\eqref{eq.def.B}, that $u \in H^2(M; E) \cap V$. Then we obtain
\begin{align}\label{eq.def.paP}
  B(u, v) \define (Pu, v) + \int_{\pa M} (\pa_{\nu}^{P} u ,v)
  d\vol_{\pa g},
\end{align}
where $d\vol_{\pa g}$ is the volume form on $\pa M$, as before. We
have $\pa_{\nu}^{P}u = e_1 a(dr, dr\otimes \nabla_\nu u) + Q$, a first
order differential operator. We shall also denote $\pa_{\nu}^{a}u
\define a(dr, dr\otimes \nabla_\nu u)$, with $u$ a section of $F_1$
over $\pa M$.  For reasons of symmetry (to have a class of operators
stable under adjoints), one may want to consider in the minimal
regularity case the following bilinear form
\begin{align*}
  B(u, v) \define B_{a}(u, v) + (Q_1u, v)_{L^2(M)} + (u, Q_2
  v)_{L^2(M)} + (Q u, v)_{L^2(\pa M)}.
\end{align*}
In that case, the boundary operator $\pa_{\nu}^{P}u$ will depend also
on $Q_2$. However, if our operators have coefficients in $W^{1,
  \infty}$, which is the case when dealing with regularity estimates,
as in this paper, then we can absorb $Q_2$ into $Q_1$ by taking
adjoints, with the price of obtaining an additional boundary term,
which, nevertheless, can then be absorbed into $Q$. See also Example
\ref{ex.pa.nu}, where the term $Q_2$ was kept in the formula.
\end{remark}

\begin{lemma}\label{lemma.e0e1}
Assume $u \in H^2(M; E)$ and let $\tilde P$ be as in Equation
\eqref{eq.def.tildeP}.  The equation
\begin{equation}\label{eq.tildeP}
  \tilde P (u) = j_{k-1}(f, h)
\end{equation}
is then equivalent to the mixed Dirichlet/Robin boundary value problem
\begin{align}\label{eq.mixed.weak}
  \left\{ \begin{aligned} P u &= f && \text{in } M\\
    e_0 u &= 0 && \text{on } \pa M\\
    e_1 \pa_{\nu}^{a} u + Q u & = h && \text{on } \pa M,
  \end{aligned} \right. 
\end{align}
where $\pa_{\nu}^{a}u \define a(dr, \nabla u)$ as above.
\end{lemma}

\begin{proof}
Indeed, $e_0u = 0$ on $\pa M$ since $u$ is in $V$, which is, by
definition the domain of $\tilde P$. Let $F \define j_{k-1}(f,
h)$. The rest follows from Equation~\eqref{eq.def.B}, which gives
\begin{align*}
  \< \tilde P(u), w \> - F(w) = & \ B(u, w) - \int_{M} (f, w)\dvol_g -
  \int_{\pa M} ( h, w)\, d\vol_{\pa g}\\
  = & \ (Pu - f, w)_M - \int_{\pa M}
  (\pa_{\nu}^{a} u + Q u - h, w)\, d\vol_{\pa g}. \qedhere
\end{align*}
\end{proof}

The boundary conditions of this lemma will be the main object of study
for us. As we will see below in Proposition \ref{prop.equiv.bc}, these
boundary conditions turn out to be, in fact, quite general. Recall the
following standard terminology.

\begin{definition}\label{def.bc}
We keep the notation of the Lemma \ref{lemma.e0e1}. We shall say, as
usual, that $e_0 u = h_0$ are the \emph{Dirichlet} boundary conditions
and $e_1 \pa_{\nu}^{a} u + Q u = h_1$ are the \emph{natural} (or
\emph{Robin}) boundary conditions.  Also, we shall say that $P$ and
the boundary conditions $(e_0, e_1 \pa_{\nu}^{a} + Q)$ of Equation
\eqref{eq.mixed.weak} are obtained from \emph{a variational
  formulation.} We shall also say that $(e_0, e_1 \pa_{\nu}^{a} + Q)$
are \emph{variational boundary conditions associated to $\tilde P$.}
\end{definition}

Non-degenerate boundary conditions are equivalent to variational ones
for suitable $a$, as we will see below.

\begin{proposition}\label{prop.equiv.bc}
Let $\tilde P$ be a second order differential operator in divergence
form associated to $a$, regarded as a bounded, measurable section of
$T^*M \otimes T^*M \otimes \End(E)$, such that $e_1 a(dr , dr)e_1$ is
invertible in $L^\infty(\pa M; \End(F_1))$. Let $C = (C_0, C_1)$ be
\emph{non-degenerate} boundary conditions. Then there is a first order
differential operator $Q$ with bounded coefficients such that the
boundary conditions $C$ and $(e_0, e_1 \pa_{\nu}^{a} + Q)$ are
equivalent. In particular, $C$ is equivalent to some variational
boundary conditions associated to $\tilde P$.
\end{proposition}

This proposition justifies our choice to consider only boundary value
problems of the form \eqref{eq.mixed.weak} instead of the general form
\eqref{eq.mixed0}. Indeed, let $\tilde P$ be the differential operator
in divergence form with boundary conditions $(e_0, e_1 \pa_{\nu}^{a} +
Q)$. Then the solutions of the equation $\tilde P (u) = F$ are in a
natural bijections to the solutions of the boundary value problem $Pu
= f$ and $Cu = h$, where $f$ and $h$ depend linearly and continuously
on $F$.

\begin{proof}
Let $C_1 = C_{10} + C_{11}\pa_{\nu}$ as explained after Equation
\eqref{eq.boundary.ops}. We have $C_0 e_0 = C_0$ with $C_0$
invertible. Then, $e_1 \pa_{\nu}^{a} + Q = e_1 a(dr, dr)e_1 \pa_{\nu} +
Q'$ and $e_1 a(dr, dr)e_1 C_{11}^{-1} C_1 = e_1 a(dr, dr)e_1
\pa_{\nu} + Q''$, where $Q'$ and $Q''$ are first order differential
operators on $F_1$ with bounded coefficients. Since $Q$ (and hence
also $Q'$) can be chosen arbitrarily with these properties, we can
certainly arrange that $Q' = Q''$ within the class of operators
considered.
\end{proof}

There is no good reason to chose $C_0$ other than $e_0$. On the other
hand, there is no reason to expect, in general, that $C_{11}$ be
invertible. However, if $C_{11}$ is \emph{not} invertible, the
behavior of the problem becomes completely different and, to the best
of our knowledge, it is not fully understood at this time (see,
however, \cite{NazarovPopoff} and the references therein).

\begin{example}\label{ex.pa.nu}
Recall the operator $\tilde P$ from Equation \eqref{eq.def.tildeP} and
the related form $B$ from Equation \eqref{eq.def.B}. Let us assume
that $M = U \subset \RR^m$ is a submanifold with boundary of dimension
$m$. (Here $\pa U$ denotes the boundary of $U$ as a manifold with
boundary, not as a subset of $\RR^m$!) Let $E = \underline{\CC}^N$,
$F_0 = \underline{\CC}^{N_0}$, and $F_1 = \underline{\CC}^{N_1}$ be
trivial bundles with $N = N_0 + N_1$---all equipped with the standard
metric. Finally, we assume that are matrix valued functions $a_{ij},
b_i, b^*_j, c, d \in L^\infty(U; M_{N})$.  We will assume that the
metric on $U$ is the euclidean metric, since this will not really
decrease the generality, but will simplify our notation.  In this
example, we choose
\begin{equation}
   V \define \{ u \in H^1(U)^N\, | \ u_1= u_2= \ldots = u_{N_1} = 0
   \mbox{ on } \pa U \} .
\end{equation}

We let $a$ denote the bilinear form on $\underline{\CC}^{mN}$
associated to the matrix $(a_{ij}\in (\underline{\CC}^{N})^*\otimes
\underline{\CC}^{N})_{ij}$. We have $\nabla_i = \pa_i$ since we are
dealing with the trivial bundles with standard metric over the
euclidean space. Let $Q = (Q^{kl})_{k, l = N_1+1}^{N}$ be a matrix
first order differential operator on $\pa U$ and $(Pu)_k$ be the $k$th
component of $Pu$. This gives for all $w\in V$
\begin{multline*}
  \< \tilde P_{(a,Q)} u, w \> \seq \int_U a(\nabla u, \nabla w)dx +
  \int_{\pa U} (Q u, w) dS \\
  = \sum_{i, j = 1}^m \sum_{k, l = 1}^{N} \int_{U} a_{ij}^{kl} \pa_j
  u_{l} \pa_i \overline{w_k} dx + \int_{\pa U} (Q u, w) dS,
\end{multline*}
where $dS$ is the induced volume form on $\pa U$. In particular,
\begin{equation*}
  (P_{(a,Q)} u)_k = (P_{a} u)_k \seq - \sum_{i, j = 1}^m \sum_{l =
    1}^{N} \pa_i(a_{ij}^{kl} \pa_j u_{l}).
\end{equation*}
Let $Q_1u\define \sum_{j=1}^m b_j \pa_j u + cu$ and $Q_2u\define
\sum_{i=1}^m \overline{b_i^*} \pa_i u$, let $u, v \in
\CI(\overline{U})^N \cap V$, and let $\nu$ be the \emph{outer, unit}
normal to $\pa U$. The formula $\int_U\, (\pa_j u)\, dx = \int_{\pa
  U}\, \nu_j u\, dS$ gives
\begin{multline*}
   \langle \tilde{P}(u), w\rangle \define \int_U a(\nabla u, \nabla w)
   dx + (Q_1u, w)_U + (u, Q_2 w)_U\\
 + \sum_{k,l = N_1 + 1}^N\int_{\pa U} Q^{kl} u_k(x) \overline{w_l(x)}
 dS
%
 = \int_U Pu(x)\overline{w(x)} dx + \int_{\pa U} \pa_{\nu}^{P}u(x)
 \overline{w(x)} dS.
\end{multline*}
Let us extend $Q$ to an $N \times N$-matrix by completing it with
zeroes. In turn, this gives that
\begin{equation*}
   P u = P_a u + \sum_{j=1}^m b_j \pa_j u + cu - \sum_{i=1}^m \pa_i(
   b_i^* u)
\end{equation*} 
and 
\begin{equation*}
  (\pa_{\nu}^{P} u)_k \define
  \begin{cases}
      \sum_{l=N_1+1}^N \Big (\sum_{ij=1}^m \nu_i a_{ij}^{kl} \pa_j
      u_{l} + \sum_{i} \nu_i b_i^{*kl} u_l \Big ) + (Q u)_k & \mbox{
        if } k > N_1\\ 0 & \mbox{ otherwise.}
  \end{cases}
\end{equation*}
 
If the coefficients $\overline{b_i^*}$ are differentiable \emph{and}
$F_1 = 0$, we can get rid of the operator $Q_2$ by absorbing it into
$Q_1$.  However, if these coefficients are not differentiable and we
want a class of operators that is closed under adjoints, then we need
to include the $Q_2$ term into the definition of $P$ (or, rather,
$\tilde P$). Moreover, if $F_1 \neq 0$, we see that $\tilde P u$
contains an additional term compared to $P u$, meaning that $\tilde P
u - Pu$ is a distribution supported on $\pa U$. Also, we see that
$\tilde P$ determines $P$, but not the other way around.
\end{example}

\subsubsection{Uniformly strongly elliptic operators}
There are some classes of operators for which the conditions of
Proposition \ref{prop.equiv.bc} are almost automatically
satisfied. Recall the following standard terminology.

\begin{definition}\label{def.u.s.e}
We shall say that the operator $\tilde P \colon V \to V^*$ defined
using the sesquilinear form $B$, see Equation \eqref{eq.def.tildeP},
is \emph{uniformly strongly elliptic} if there exists $c_a > 0$ such
that
\begin{equation*}
	\Re  \big ( a( \eta \otimes \xi , \eta \otimes \xi ) \big ) 
	\ge c_a \|\eta\|^2  \|\xi\|^2
\end{equation*}
for all $x$ and all $\xi\in E_x$ and $\eta \in T_x^* M$.  The
associated operator $P$ will be called \emph{uniformly strongly
  elliptic.}  A family of operators is called \emph{uniformly strongly
  elliptic} if each operator is uniformly strongly elliptic and we can
choose the same constant $c_a$ for all operators in the family. We
shall say that $\tilde P$ is \emph{uniformly elliptic} if there exists
$c_e > 0$ such that, for every $\xi \in E_x$, there exists $\xi_1 \in
E_x$, $\|\xi_1\| = \|\xi\|$, such that
\begin{equation*}
	\big | a( \eta \otimes \xi , \eta \otimes \xi_1 ) \big | \ge
        c_e \|\eta\|^2 \|\xi\|^2
\end{equation*}
for all $x$ and $\eta \in T_x^* M$.
\end{definition}

Note that in the above definition, we have taken advantage of the fact
that our operator acts on sections of the same bundle. For operators
acting between sections of different bundles, the definition will
change in an obvious way (replacing the injectivity of the principal
symbol with its invertibility).  We see that if $P$ is uniformly
strongly elliptic, then $e_1 a(\nu, \nu)e_1$ is invertible in
$L^\infty(\pa M; \End(F_1))$. (If $e_1 = 1$, i.e. if $F = E\vert_{\pa
  M}$, it is enough to assume that $a$ is uniformly elliptic. We agree
that if $F_1$ is the zero bundle on some component of $\pa M$, then we
consider every endomorphism of it to be invertible on that set.)

\subsubsection{The scale of regularity for boundary value problems}
Let $P$ be associated to $B$ as in Equation \eqref{eq.def.tildeP} and
$V$ be as in Equation \eqref{eq.def.VnW}.  We define
\begin{equation*}
 \check H^{\ell-1}(M; E) \define \left\{
 \begin{aligned}
   & H^{\ell-1}(M; E) \oplus H^{\ell-1/2}(\pa M; F_1) && \text{for }
   \ell \geq 1\\
   & V^{*} && \text{for } \ell = 0.
 \end{aligned}\right.
\end{equation*}
The natural exact sequence $0 \to H^{-1/2}(\pa M; F_1) \to V^* \to
H^{-1}(M; E) \to 0$, where the second map is induced by the trace map,
see Theorem~\ref{thm.trace}, shows that we have a natural scale of
regularity spaces.  In particular, $\check H^{\ell+1}(M; E) \subset
\check H^{\ell}(M; E)$, $\ell \geq0$. In general, the natural
inclusion $\check H^{\ell}(M; E) \to \check H^{-1}(M; E) \define
V^{*}$ is given by the operators $j_\ell$ defined in Equation
\eqref{eq.def.jk}. Let $\tilde P_\ell \colon H^{\ell+1}(M; E) \cap V
\to \check H^{\ell-1}(M; E)$ be given by
\begin{equation}\label{eq.def.Pm}
 \tilde P_\ell (u) \define ( Pu,\ e_1\pa_{\nu}^{a} u + Q u) \ \mbox{
   for } \ell \geq 1.
\end{equation}
Then the relation between $\tilde P_\ell$ and $\tilde{P}$ is by
\eqref{eq.tildeP} expressed in the commutativity of the diagram
\begin{equation}\label{eq.diagram}
\begin{CD}
  H^{\ell+1}(M; E) \cap V @>{\tilde P_\ell}>> \check H^{\ell-1}(M; E) \\
  @VVV @VV{j_{\ell-1}}V \\
  H^{1}(M; E) \cap V @>{\tilde P}>> \check H^{-1}(M; E)
\end{CD}
\end{equation}
where the vertical arrows are the natural inclusions.  Thus, although
the definition of $\tilde P_0 \define \tilde P \colon H^{1}(M; E) \cap
V \to \check H^{-1}(M; E)$ is different from that of $\tilde P_\ell$
for $\ell >0$, it fits into a scale of regularity spaces.

\subsection{Regularity conditions}
The scale of regularity spaces provides a good setting to study
regularity conditions

\begin{notation}\normalfont{
We shall denote by $\maD^{\ell,j}(M;E)$ the set of pairs $(D,C)$,
where $D$ is a second order differential operator defined on sections
of $E\to M$ and $C$ is an order $j$ boundary condition, with both $D$
and $C$ assumed to have coefficients in $W^{\ell,\infty}$. In case $M$
has no boundary (and thus there are no boundary conditions), we shall
denote the resulting space $\maD^{\ell,\emptyset}(M;E)$. If $E \ede
\underline{\CC}$ (that is, if we are dealing with \emph{scalar}
boundary value problems), then we shall drop the vector bundle from
the notation.}
\end{notation}

The assumption that $C$ is an order $j$ boundary condition implies
that $E\vert_{\pa M} = F_j$, and hence that $F_{1-j} = 0$. The general
case just involves a more complicated notation. We endow the space
$\maD^{\ell,j}(M;E)$ with the norm defined by the maximum of the
$W^{\ell,\infty}$-norms of the coefficients.  Recall the following
definition:

\begin{definition}\label{def.loc.reg}
We say that $(D, C) \in \maD^{\ell,j}(M;E)$, $\ell\geq j + 1$,
satisfies an \emph{$H^{\ell+1}$-regularity estimate} on $M$ if there
exists $c_R > 0$ with the following property: For any $w \in H^{k}(M;
E)$, $\ell\geq k\geq j + 1$, with compact support in $M$ such that $D
w \in H^{k-1}(M; E)$ and $C w\in H^{k-j+1/2}(\pa M; E)$, we have $w
\in H^{k+1}(M; E)$ and
\begin{equation*}
  \| w \|_{H^{k+1}(M; E)}\ \leq\ c_R \left(\|D w\|_{H^{k-1}(M; E)} +
  \| w \|_{H^{k}(M; E)} + \|C w\|_{H^{k-j+1/2}(\pa M; E)}\right).
\end{equation*} 
If $D \in \maD^{\ell,\emptyset}(M;E)$, we just drop the term $\|C
w\|_{H^{k-j+1/2}(\pa M; E)}$. If $j = 1$ and if $(D, C)$ are obtained
from a variational formulation, then we allow also $\ell = 1$, with
the relations $D w \in L^2(M; E)$ and $C w\in H^{1/2}(\pa M; E)$ being
replaced with $\tilde D w = j_0(f, h)$, with $(f, h) \in L^2(M; E)
\oplus H^{1/2}(\pa M; E)$ and the estimate is replaced with
\begin{equation*}
  \| w \|_{H^{2}(M; E)}\ \leq\ c_R \left(\|f\|_{L^2(M; E)} +
  \| w \|_{H^{1}(M; E)} + \|h\|_{H^{1/2}(\pa M; E)}\right),
\end{equation*}
in which case we also say that $\tilde D$ satisfies an
\emph{$H^{\ell+1}$-regularity estimate} on $M$.
\end{definition}

The following comments on this definition are in order.

\begin{remark}\label{rem.reg}
\begin{enumerate}[(i)]
\item The condition that $w$ have compact support in $M$ \emph{does
  not} imply that it vanishes near the boundary $\pa M$.

\item Typically, we shall consider boundary value problems coming from
  a variational formulation, but, at this time, it is not necessary to
  make this assumption.

\item The conditions that $D$ and $C$ have coefficients in $W^{\ell,
  \infty}$ are needed so that $D w \in H^{k-1}(M; E)$ and $C w \in
  H^{k-j+1/2}(\pa M; E)$ for $w \in H^{k+1}(M; E)$, $k \leq \ell$.

\item Let $(D', C') \in \maD^{\ell,j}(M;E)$ be a second boundary value
  problem such that $D - D'$ has order $< 2$ and $C - C'$ has order $<
  j$. Then $(D, C)$ satisfies an $H^{\ell+1}$-regularity estimate on
  $M$ if, and only if, $(D', C')$ does so.
\end{enumerate}
\end{remark}

\begin{remark}\label{rem.higher.order}
The careful reader might have noticed already that, apart from
notation, we have not really used the fact that our operator is second
order, except maybe in Proposition \ref{prop.equiv.bc}. For instance,
for an order $2m$ problem, one would consider
\begin{equation*}
  a \in L^\infty(M; T^{\otimes m}M \otimes T^{\otimes m}M
  \otimes \End(E)),
\end{equation*}
but the form $B$ would involve several boundary terms relating the
normal derivatives at the boundary. For $V$ we would take an
intermediate subspace $H_0^{m}(M; E) \subset V \subset H^m(M;E)$. This
could be given by a sub-bundle $F_1 \subset E_{\pa M}^{m}$
representing the possible values of $(u, \pa_\nu u, \ldots,
\pa_\nu^{m-1}u)$ at the boundary:
\begin{equation*}
  V := \{ u \in H^{m}(M; E) \, \vert \ (u, \pa_\nu u, \ldots,
  \pa_\nu^{m-1}u) \in L^{\infty}(\pa M; F_1) \}.
\end{equation*}
Nevertheless, non-local conditions are also important
\cite{LionsMagenes2}. One could then proceed to deal with higher order
problems as in \cite{LionsMagenes1}, for instance. Since Proposition
\ref{prop.equiv.bc} plays an important role in justifying our choice
of \emph{essentially always} using variational boundary value
conditions and in view of the many complications that arise when
dealing with higher order problems (see \cite{ADN1, LionsMagenes1,
  Pryde, SchectherGBVP}), we decided to restrict ourselves to the case
of second order operators. See also Remark~\ref{rem.in.general}.
\end{remark}

\section{Uniform regularity estimates for families}\label{sec.u.est}

We shall consider the same framework as in the previous section, in
particular, $(M,g)$ will continue to be a Riemannian manifold with
smooth boundary $\pa M$ such that $TM$ has totally bounded curvature.
Also, we continue to assume that $E \to M$ has totally bounded
curvature.

\subsection{Compact families of boundary value problems}
We use the same notation as in the previous section. In particular,
$B$, $\tilde P$, $P$, $a$, $e_1 \pa_{\nu}^{a} + Q$, $C = (C_0, C_1)$,
and $F_0 \oplus F_1 = E\vert_{\pa M}$ are as in the previous
section. In particular, $\tilde P$ will always represent a second
order differential operator in divergence form associated to the
sesquilinear form $a$; in particular, (the quadratic map associated
to) $a$ is the principal symbol of $P$.
  
We shall need a \emph{uniform} version of Definition
\ref{def.loc.reg}.

\begin{definition}\label{def.unif.reg}
Let $S\subset\maD^{\ell,j}(M;E)$ and $N \subset M$ a submanifold with
boundary $\pa N = \pa M \cap N$ of $M$. We shall say that $S$
satisfies a \emph{uniform $H^{\ell+1}$-regularity estimate on $N
  \subset M$} if each $(D, C)\in S$ satisfies an
$H^{\ell+1}$-regularity estimate on $N\subset M$ with a bound $c_R$
(in Definition~\ref{def.loc.reg}) independent of $(D, C) \in S$. If
$S\subset\maD^{\ell,\emptyset}(M;E)$, we just consider $D \in S$.
\end{definition}

Note that $S$ in the above definition is \emph{not} assumed to be
bounded. However, typically, we shall obtain the independence of the
bound $c_R$ by assuming also that $S$ is compact (or, sometimes, just
precompact).

\begin{proposition}\label{prop.unif.reg}
Assume that $S\subset \maD^{\ell,j}(M;E)$ is \emph{compact} and that
each $(D,C)$ satisfies an $H^{\ell+1}$-regularity estimate on
$M$. Also, assume that the restriction $H^j(M; E) \to
H^{j-1/2}(\partial M; E)$ is continuous for all $j\geq 1$.  Then $S$
satisfies a \emph{uniform} $H^{\ell+1}$-regularity estimate on
$M$. The same result holds if $S\subset \maD^{\ell,\emptyset}(M;E)$.
\end{proposition}

The assumption that the restriction (trace map) $H^j(M; E) \to
H^{j-1/2}(\partial M; E)$ is continuous for all $j\geq 1$ is, of
course, satisfied if $M$ is an open subset of a manifold with boundary
and bounded geometry $M_0$ \cite{GrosseSchneider} (see Theorem
\ref{thm.trace}). In particular, this is the case if $M$ is the domain
of a Fermi coordinate chart. Also, note that here the boundary is in
the sense of manifolds with boundary, and not in the sense of
point-set topology.

\begin{proof} Consider first the case $S\subset \maD^{\ell,j}(M;E)$.
Let us assume the contrary and show that we obtain a contradiction.
That is, let us assume that there exist sequences $(D_i,C_i)\in S$ and
$ w_i\in H^{k+1}(M; E)$, with compact support in $M$, such that
\begin{equation*}
  \| w_i \|_{H^{k+1}(M; E)}\ > \ 2^i \left( \|D_{i}
  w_i\|_{H^{k-1}(M; E)} + \| w_i \|_{H^{k}(M; E)} + \|C_{i}
  w_i\|_{H^{k-j+1/2}(\pa M; E)} \right).
 \end{equation*}
Since $S$ forms a compact subset in $\maD^{\ell,j}(M;E)$, by replacing
$(D_i, C_i)$ with a subsequence, if necessary, we can assume that
$(D_{i}, C_{i})$ converges.  Let us denote the limit with $(D,C)\in
S$.  Thus, there is a sequence $\epsilon_i\to 0$ with
\begin{align*}
  \|D_{i} w\|_{H^{k-1}(M;E)} & \ \geq\ \|D w\|_{H^{k-1}(M; E)}
  -\epsilon_i \| w\|_{H^{k+1}(M; E)},\\
  \|C_{i} w\|_{H^{k-j+\frac{1}{2}}(\partial M;E)} &\ \geq\ \|C
  w\|_{H^{k-j+\frac{1}{2}}(\partial M; E)} -\epsilon_i \|
  w\|_{H^{k+\frac{1}{2}}(\partial M; E)}.
  \end{align*}  
Together with the assumed continuity of the trace map, this implies
that there exists $c' > 0$ such that
  \begin{multline*}
  \| w_i \|_{H^{k+1}(M; E)} \ > \ 2^i \left( \|D w_i\|_{H^{k-1}(M; E)}
  + \| w_i \|_{H^{k}(M; E)} + \|C w_i\|_{H^{k-j+1/2}(\pa M;
    E)}\right.\\
  \left.  -c'\epsilon_i \| w_i \|_{H^{k+1}(M; E)}\right),
 \end{multline*}
for all $i$.  On the other hand, $(D,C)$ satisfies, by assumption, an
$H^{k+1}$- regularity estimate. Consequently, there is a $c >0$ such
that
  \begin{equation*}
  \|D w_i\|_{H^{k-1}(M; E)} + \| w_i \|_{H^{k}(M;E)} + \|C
  w_i\|_{H^{k-j+1/2}(\pa M; E)}\geq c^{-1} \| w_i \|_{H^{k+1}(M;
    E)}\ ,
 \end{equation*}
  and hence we obtain
 \begin{equation*}
  \| w_i \|_{H^{k+1}(M; E)} \geq 2^i(c^{-1}-c'\epsilon_i) \| w_i
  \|_{H^{k+1}(M; E)}\,.
 \end{equation*}
For $i\to \infty$, this gives the desired contradiction since
$2^i(c^{-1}- c' \epsilon_i) \to \infty$, for $i \to \infty$, and $\|
w_i \|_{H^{k+1}(M; E)} \neq 0$.  This completes the proof if $S\subset
\maD^{\ell,j}(M;E)$.

If $S\subset \maD^{\ell,\emptyset}(M;E)$, the proof is obtained by
simply dropping the terms that contain $C w$ from the above proof. (We
can further replace $2 \epsilon_i$ with $\epsilon_i$, but that is not
essential.)
\end{proof}

Recall that a \emph{relatively compact} subset is a subset whose
closure is compact.

\begin{proposition}\label{prop.limit.reg} 
Let $N\subset M$ be a relatively compact open subset and let $S\subset
\maD^{\ell+1,j}(M;E)$ be a \emph{bounded} subset.  Assume that every
$\maD^{\ell,j}(N;E)$--limit $(\tilde{D},\tilde{C})$ of a sequence
$(D_i, C_i)\in S$ satisfies an $H^{\ell+1}$-regularity estimate on
$N$.  Then $S$ satisfies a uniform $H^{\ell+1}$-regularity estimate on
$N$.  The same result holds if $S\subset
\maD^{\ell+1,\emptyset}(M;E)$.
\end{proposition}

We remark that in this proposition the compactness condition of
Proposition~\ref{prop.unif.reg} is replaced by a \emph{higher
  regularity assumption} on the coefficients. This is needed in order
to use the Arzela-Ascoli Theorem.  Moreover, we note that by choosing
a constant sequence, we see that the assumptions imply that each
element in $S$ satisfies an $H^{\ell+1}$-regularity estimate on
$N$. We also note that $(\tilde{D},\tilde{C})$ in the statement
automatically satisfies $(\tilde{D},\tilde{C}) \in
\maD^{\ell,j}(N;E)$.

\begin{proof}[Proof of Proposition~\ref{prop.limit.reg}]
We treat explicitly only the case $S\subset \maD^{\ell+1,j}(M;E)$, the
case $S\subset \maD^{\ell+1,\emptyset}(M;E)$ being completely similar.
Since the coefficients of all boundary value problems in $S$ are
bounded in $W^{\ell+1,\infty}(M)$, the set of coefficients of the
operators $(D, C)\in S$ is precompact in $W^{\ell,\infty}(N)$, by the
Arzela-Ascoli theorem.  Let $K$ be the closure of $S$ in $\maD^{\ell,
  j}(N, E)$, which is therefore a compact set. Moreover, our
assumptions imply that every element in $K$ satisfies an
$H^{\ell+1}$-regularity estimate.  Proposition~\ref{prop.unif.reg}
then implies the result.
\end{proof}

We shall need the following lemma.

\begin{lemma}\label{lemma.restriction}
Let $S \subset {\maD^{\ell, j}(M; E)}$ satisfy a \emph{uniform}
$H^{\ell+1}$-regularity estimate on $M$. If $N \subset M$ is an open
subset, then $S$ satisfies a uniform $H^{\ell+1}$-regularity estimate
on $N$. The same result holds in the boundaryless case.
\end{lemma}

\begin{proof}
Let $w \in H^k(N; E)$ with compact support. Then $w \in H^k(M; E)$ and
the uniform regularity estimate for $w$ on $M$ yields the desired
result.
\end{proof}

\begin{remark}
To deal with the general case when $C$ is not a boundary condition of
fixed order $j$, but rather combines a boundary condition $C_1$ of
order one and a boundary condition $C_0$ of order zero, we just
replace $\|C w\|_{H^{k-j+1/2}(\pa M; E)}$ with $\|C_1
w\|_{H^{k-1/2}(\pa M; E)} + \|C_0 w\|_{H^{k+1/2}(\pa M; E)}$. In this
case, we may have that both $F_0 \neq 0$ and $F_1 \neq 0$.
\end{remark}

\subsection{Higher regularity and bounded geometry}
\label{ssec.hr}

The relevance of uniform regularity conditions introduced in
Section~\ref{sec.u.est} is that it allows us to obtain higher
regularity on manifolds with boundary and bounded geometry and
suitable boundary conditions as follows.

Let $(M,g)$ be a Riemannian manifold with boundary and bounded
geometry, as before. Let $(P, C)$ be a boundary value problem on
$M$. We shall assume that $(P, C)$ comes from a variational
formulation, since most of the results will require this
assumption. (This means that if $C = (C_0, C_1)$, then $\tilde P$ can
be identified with $(P, C_1)$ and $C_0$ is simply the projection onto
$F_0$ at the boundary.) We assume, for notational simplicity, that $C$
has constant order $j$ at the boundary. For the same reasons, we also
assume $P$ acts and takes values in the same vector bundle $E$ (that
is, $E_1 = E$), as before.  The results of this subsection hold,
however, in full generality (when the vector bundles $E$, $E_1$, and
$F$ of Section~\ref{sec.u.est} are distinct), but with some obvious
changes.  Let $0 < r \leq r_{FC}$, as in Definition~\ref{FC-chart}.
Recall that $U_p$ and $\kappa_p$ from \eqref{eq.Ugamma} and $\xi_p$
from Definition~\ref{def_sync} are such that either $p \in \pa M$ or
$\dist(p, \pa M) \geq r$. We denote by $(P_p, C_p)$ the induced
boundary value problems on $\kappa_p^{-1}(U_p) = B_{2r}^{m}(0)\times
[0,2r)$, if $p \in \pa M$. Then $P_p= \xi_p^*\circ P\circ (\xi_p)_*$
  and $C_p= \xi_p^*\circ C\circ (\xi_p)_*$, with the obvious notation,
  meaning that the operators correspond through the diffeomorphisms
  $\xi_p$. If $\dist(p, \pa M) \geq r$, there is no $C_p$ and we
  obtain a differential operator $P_p$ on $B_{r}^{m+1}(0)$. Let $t$
  denote the rank of $E$ and let
\begin{equation}
 \begin{lgathered}\label{eq.def.F}
 \maF_b \define \{(P_p, C_p) \vert\ p \in \pa M \} \subset \maD^{0,
   j}(B_{2r}^{m}(0)\times [0,2r); \CC^t)\\
 \maF_i \define \{(P_p) \vert\ \dist(p, \pa M) \geq r \} \subset
 \maD^{0, \emptyset}(B_{r}^{m+1}(0); \CC^t),\
\end{lgathered}
\end{equation}
be the induced \emph{boundary} and \emph{interior} families of
operators. Note that we always equip $B_{2r}^{m}(0)\times [0,2r)$,
  (respectively, $B_{r}^{m+1}(0)$) with the euclidean metric.

\begin{theorem}\label{thm.reg0}
Let $(M,g)$ be a Riemannian manifold with boundary and bounded
geometry and let $E\to M$ be a Hermitian vector bundle with totally
bounded curvature.  Let $(P, C)\in \maD^{\ell,j}(M; E)$. If each of
the families $\maF_b \define \{(P_p, C_p)| \ p \in \pa M\}$ and
$\maF_i \ede \{P_p| \ \dist(p, \pa M) \geq r\}$ (see Equation
\ref{eq.def.F} and above) satisfy a uniform $H^{\ell+1}$-regularity
estimate, then $(P,C)$ satisfies an $H^{\ell+1}$-regularity estimate.
The converse is also true.
\end{theorem}

\begin{proof}
This follows from Definition~\ref{def.loc.reg} of uniform order $k$
regularity estimates, from Proposition~\ref{prop.part.unit}, and from
Lemma \ref{lemma.commutator}. Let us choose the $r$-uniform partition
of unity $\phi_\gamma$ used in those results such that $\pa_{\nu}
\phi_\gamma$ vanishes at the boundary. This can be done first by
choosing an $r$-uniform partition of unity on $\pa M$.
\begin{align*}
 \Vert u\Vert_{H^{k+1}}^2\lesssim& \sum_\gamma \Vert \xi_\gamma^*
 (\phi_\gamma u)\Vert_{H^{k+1}}^2\\
 \lesssim& \sum_\gamma \left( \Vert P_\gamma \xi_\gamma^* (\phi_\gamma
 u)\Vert_{H^{k-1}} +\Vert C_\gamma \xi_\gamma^* (\phi_\gamma
 u)\Vert_{H^{k-j+\frac{1}{2}}} + \Vert \xi_\gamma^* (\phi_\gamma
 u)\Vert_{H^{k}}\right)^2\\
 \lesssim & \sum_\gamma \left( \Vert \xi_\gamma^* P (\phi_\gamma
 u)\Vert_{H^{k-1}}^2 +\Vert \xi_\gamma^* C (\phi_\gamma
 u)\Vert_{H^{k-j+\frac{1}{2}}}^2 + \Vert \xi_\gamma^* (\phi_\gamma
 u)\Vert_{H^{k}}^2\right)\\
 \lesssim& (\Vert P u\Vert_{H^{k-1}}^2 +\Vert C
 u\Vert_{H^{k-j+\frac{1}{2}}}^2 + \Vert u\Vert_{H^{k}})^2 +
 \sum_\gamma \Vert [P,\phi_\gamma] u\Vert^2_{H^{k-1}} \\
  & + \sum_\gamma \Vert [C,\phi_\gamma] u\Vert^2_{H^{k-j+\frac{1}{2}}}
 \,,
\end{align*}
since the trivializations $\xi_\gamma$ have uniformly bounded norms.
Next we notice that $\sum_\gamma \Vert [P,\phi_\gamma]
u\Vert^2_{H^{k-1}}\lesssim \Vert u \Vert^2_{H^{k}}$ since the family
$\phi_\gamma$ is uniformly locally finite and by
Lemma~\ref{lemma.commutator}. The boundary term is treated similarly:
 $\sum_\gamma \Vert [C,\phi_\gamma]
u\Vert^2_{H^{k-j-\frac{1}{2}}} \lesssim \Vert u
\Vert^2_{H^{k-\frac{1}{2}}} \lesssim \Vert u \Vert^2_{H^{k}}$.  Here
the last inequality is given by the trace theorem. For the first
inequality we note that for $j=0$ the commutator is actually zero. For
$j=1$ we have $[C,\phi_\gamma]=[C_{10},\phi_\gamma]$ since
multiplication by $\phi_\gamma$ and $\pa_{\nu}$ commute, given the
product form of $\phi_\gamma$ near the boundary.  Together with Lemma
\ref{lemma.commutator}, this completes the proof.
\end{proof}

\begin{remark}
The method of proof of Theorem~\ref{thm.reg0} will yield similar
global results in other classes of spaces, as long as the local
regularity results are available and as long as a local description of
these spaces using partitions of unity is available. This is the case
for the $L^p$-Sobolev spaces, $1 < p < \infty$, for which we have both
the local description using partitions of unity (Proposition
\ref{prop.part.unit}) and the local regularity results
\cite{TrudingerBook}.
\end{remark}

\subsection{Regularity for Dirichlet boundary conditions}
The results of the previous two sections were tailored to deal with
the Dirichlet boundary conditions. It is known \cite{Agranovich07,
  AgranovichBook, Taylor1} that strongly elliptic operators with
Dirichlet boundary conditions satisfy regularity conditions.  We
formulate this well-known result as a lemma for further use. As
before, $E \to M$ will be a vector bundle with bounded geometry.

\begin{lemma}\label{lemma.euclidean}
Let $(P, C) \in \maD^{\ell, 0}(B^m_r(0) \times [0, r); E)$, $0 < r\leq
  \infty$, be a uniformly strongly elliptic boundary value problem
  with \emph{Dirichlet boundary conditions.} Then $P$ satisfies an
  $H^{\ell+1}$-regularity estimate on $M$. The same result holds for a
  uniformly elliptic operator $P \in \maD^{\ell, \emptyset}(B^m_r(0);
  E)$, $0 < r\leq \infty$.
\end{lemma}

From this we obtain

\begin{corollary}\label{cor.Dir}
Let $S \subset \maD^{\ell+1, 0}(B^m_r(0) \times [0, r); E)$ be a
  \emph{bounded, uniformly strongly elliptic family} of boundary value
  problems on $B^m_r(0) \times [0, r)\subset \RR^{m+1}$ equipped with
    the euclidean metric, $r \leq \infty$. We assume that the boundary
    conditions are all Dirichlet.  Then the family $S$ satisfies a
    uniform $H^{\ell+1}$-regularity estimate on $B^m_{r'}(0) \times
    [0, r')$, $r'<r$.
\end{corollary}

\begin{proof} 
Let $(D_n, C_n) \in S$ converge to $(D, C) \in \maD^{\ell,
  j}(B^m_{r'}(0) \times [0, r'); E)$. Then $D$ is a uniformly strongly
  elliptic operator because the parameter $c_a$ stays away from 0 on
  $S$, in view of Definition~\ref{def.u.s.e}.
  Lemma~\ref{lemma.euclidean}, then gives that $(D, C)$ satisfies an
  $H^{\ell+1}$-regularity estimate, since the type of boundary
  conditions (Dirichlet or Neumann) do not change by taking
  limits. This allows us to use Proposition~\ref{prop.limit.reg} for
  the relatively compact subset $N \define B^m_{r'}(0) \times [0, r')$
    of $M \define B^m_{r}(0) \times [0, r)$ to obtain the result.
\end{proof}

Analogously, (in fact, even more directly, since we do not have to take
boundary conditions into account), we obtain

\begin{corollary}\label{cor.unif.elliptic}
Let $S \subset \maD^{\ell+1, \emptyset}(B^{m+1}_r(0); E)$ be a bounded
uniformly elliptic family of differential operators on $B^{m+1}_r(0)
\subset \RR^{m+1}$, for some $0 < r\leq \infty$.  Then the family $S$
satisfies a uniform $H^{\ell+1}$-regularity estimate on
$B^{m+1}_{r'}(0) \subset \RR^{m+1}$ for any $r' < r$.
\end{corollary}

\begin{remark}
The regularity results of this section extend to the $L^p$-Sobolev
spaces $W^{\ell, p}$, $1 < p < \infty$, with essentially the same
proofs by using also the results in \cite{TrudingerBook}.
\end{remark}

Note that in Corollaries \ref{cor.Dir} and \ref{cor.unif.elliptic} we
use a slightly stronger assumption on the coefficients of our
operators than usually, namely we require them to have $W^{\ell +1 ,
  \infty}$-regularity (usually we require only $W^{\ell,
  \infty}$-regularity). This is required since we will use Proposition
\ref{prop.limit.reg}. Combining these result, we obtain the following.

\begin{theorem}\label{thm.reg.D}
Let $P$ be a uniformly strongly elliptic second order differential
operator with coefficients in $W^{\ell+1, \infty}$ acting on sections
of $E \to M$.  Then there exists $c > 0$ such that, if $u \in
H^{\ell}(M; E)$, $Pu \in H^{\ell-1}(M; E)$, and $u\vert_{\pa M} \in
H^{\ell+1/2}(\pa M; E)$, then $u \in H^{\ell+1}(M; E)$ and
 \begin{equation*}
   \| u \|_{H^{\ell+1}(M; E)} \leq c \left( \|P u\|_{H^{\ell-1}(M; E)} + \| u
   \|_{H^{\ell}(M; E)} + \| u \|_{H^{\ell + 1/2}(\pa M; E)} \right).
  \end{equation*}
\end{theorem}

\begin{proof} 
Corollaries \ref{cor.Dir} and \ref{cor.unif.elliptic} show that the
assumptions of Theorem~\ref{thm.reg0} are satisfied. That theorem
immediately gives our result.
\end{proof}

An alternative proof of this result is obtained using the methods of
Section~\ref{sec.coercive}. The advantage of the method used in this
section is that it applies right away to higher order equations.

\section{A uniform Shapiro-Lopatinski regularity condition}\label{sec.SL}
The case of Neumann boundary conditions seems to be different (at
least for systems, see \cite{AgranovichBook, Agranovich07, Taylor1,
  Taylor2}).  This case, as well as that of Robin boundary conditions
motivates, in part, the results of this section. In particular, we
introduce a uniform version of the Shapiro-Lopatinski condition
\cite{ADN1, AgranovichBook, Agranovich07, DaugeBook, Hormander3,
  Lopatinski, NP, Taylor1}, which turns out to characterize the
operators satisfying regularity. Our approach has several points in 
common to the ones in \cite{LionsMagenes1, Pryde}. To deal with the concrete case of
Robin (and Neumann) boundary conditions, we use positivity to check
that the uniform Shapiro-Lopatinski regularity conditions are
satisfied.

\subsection{Homogeneous Sobolev spaces and regularity conditions}
We need first the following homogeneous (with respect to dilations)
versions of the usual Sobolev spaces. This setting was used for
similar purposes in \cite{Pryde}. For simplicity, we work in
  $\RR^n$, but the same considerations apply to any vector space $V$
  endowed with a metric (or a half-vector space $V_+ \subset V$).
  (See, however, Equation \eqref{eq.conf.inv} for the dependence of
  the norms on the choice of the metric.)

Let $\hat u$ be the Fourier transform of $u$, regarded as a tempered
distribution (the normalizations are not important here). Consider the
\emph{semi-norm}
\begin{equation}
 |u|_{H^s(\RR^n)}^2 \define \int_{\RR^n} |\xi|^{2s} |\hat u(\xi)|^2
 d\xi,
\end{equation}
and $\cdotH{s}(\RR^n)\define \{ u\, \vert \ |u|_{H^s(\RR^n)} < \infty
\}$. Here $u$ is such that $|\xi|^{s} |\hat u(\xi)|$ is a function
(but $|\hat u(\xi)|$ is \emph{not assumed} to be a function, which
allows us to include polynomials of low degree in $\cdotH{s}(\RR^n)$).
When $s \in \ZZ_+$, the seminorm $|u|_{H^s(\RR^n)}$ is equivalent to
the (usual) $\sum_{|\alpha| = s} \|\pa^\alpha u\|_{L^2(\RR^n)}$, which
allows us to define in this case also $|u|_{H^s(\RR^n_+)}$ for a
function $u$ defined only on a half-space of the form
$\RR^{n}_+\define \RR^{n-1} \times [0, \infty)$. In what
    follows, we shall write simply $|\cdot|_{H^s}$ for the above
    semi-norms. The reason for considering the semi-norms $|\cdot
  |_{H^s}$ and the spaces $\cdotH{s}(\RR^n)$ is that they have good
  dilation properties (see the next lemma). These definitions extend
  right away to functions with values in $\CC^N$ yielding the spaces
  $\cdotH{s}(\RR^n; \CC^N) = \cdotH{s}(\RR^n)^N$.

\begin{lemma} \label{lemma.classical} 
Let $s > 0$.
\begin{enumerate} [(i)]
\item Let $\alpha_t (f) (s) = f(ts)$, then $|\alpha_t(f) |_{H^s} =
  t^{s-n/2} |f|_{H^s}$.

\item $\lim_{t \to \infty} t^{-s + n/2} \|\alpha_t(f) \|_{H^s} =
  |f|_{H^s}$.

\item If $T$ is an order $k$, homogeneous, constant coefficient
  differential operator, then it defines continuous maps $T \colon
  \cdotH{s}(\RR^n) \to \cdotH{s-k}(\RR^n)$.

\item If $s \in \NN$, the restrictions $\cdotH{s}(\RR^n) \to
  \cdotH{s}(\RR^n_+)$ and $\cdotH{s}(\RR^n) \to
  \cdotH{s-1/2}(\RR^{n-1})$ are continuous and surjective and, if also
  $T$ is as in (iii) and has order $k \leq s$, $T \colon
  \cdotH{s}(\RR^n_+) \to \cdotH{s-k}(\RR^n_+)$ is continuous.
\end{enumerate} 
\end{lemma}

\begin{proof}
The proof is standard \cite{EvansBook, JostBook, LionsMagenes1,
  Taylor1} and easy. We include a few details for the benefit of the
reader.

(i), (ii), and (iii) are direct calculations based on the definition
and the formulas for the norms in terms of the Fourier
transform. 

(iv) is slightly less trivial, but it is known in the classical case
and our case can either be proved directly following the same method
as in the case of the classical Sobolev spaces or it can be reduced to
the classical Sobolev spaces using (ii). Indeed, for the continuity,
this is immediate, as the restriction maps commute with $\alpha_t$.
For the surjectivity, one has to argue also that there exist right
inverses for the restriction that are $\alpha_t$-invariant. This is
done by choosing a partition of unity that is invariant with respect
to $\alpha_2$ and then choosing a right inverse to the restriction on
one coordinate patch that is continuous for the classical norms. We
replicate this right inverse for all patches using $\alpha_2$. This
yields a continuous right inverse $H^{s-1/2}(\RR^{n-1/2}) \to
H^{s}(\RR^n)$ for the restriction $H^{s}(\RR^n) \to
H^{s-1/2}(\RR^{n-1/2})$ that is, moreover, invariant for
$\alpha_2$. We obtain a fully $\RR^*_+$-invariant inverse by
integrating over the compact group $\RR^*_+/2^{\ZZ}$. This right
inverse will work also for the homogeneous spaces.
\end{proof}
 
Let $D$ be an $N \times N$ matrix of second order differential
operators on $B_r^m(0) \times [0, r)$, and let $C$ be a boundary
  differential operator of order $j$. Assume the coefficients have
  $W^{\ell, \infty}$ smoothness, that is, that $(D, C) \in \maD^{\ell,
    j}(B_r^m(0) \times [0, r); \CC^N)$.  Here $r > 0$, and the case $r
    = \infty$ is not excluded.

\begin{definition} \label{def.SL.reg}
Let $\ell\geq j+1$.  We shall say that $(D, C) \in \maD^{\ell,
  j}(B_r^m(0) \times [0, r); \CC^N)$ satisfies an
  \emph{$\cdotH{\ell+1}$-regularity estimate} on $B_r^m(0) \times [0,
    r)$ if there exists $c_{SL}$ such that for all $\ell\geq k\geq
    j+1$, we have
\begin{equation*}
  | w |_{H^{k+1}(\RR^{m+1}_+)}\ \leq\ c_{SL} \left(|D w|_{H^{k-1}
    (\RR^{m+1}_+)} + |C w|_{H^{k-j+1/2}(\RR^m)}\right),
\end{equation*} 
for all $w$ smooth with compact support in $B_r^m(0) \times [0, r)$,
where we have removed the vector bundle from the notation for the
norms. If $(D,C)$ are obtained from a variational formulation, then we
allow also $\ell = k = j = 1$ and in that case we assume that $\tilde
D w = j_0(f, h)$ and we replace the right hand side with $|f|_{L^{2}
  (\RR^{m+1}_+)} + |h|_{H^{1/2}(\RR^m)}$. This case is, in fact,
crucial in applications, since it is the one obtained using coercivity
to prove well-posedness.
\end{definition}

This definition is very similar to Definition~\ref{def.loc.reg},
except that we consider \emph{semi-norms} instead of norms.  Also, we
only require $w$ to be smooth. However, an operator satisfying the
conditions in Definition~\ref{def.SL.reg}, will satisfy also those of
Definition~\ref{def.loc.reg}. We formulate this result as a lemma, for
further use.

\begin{lemma}\label{lemma.dotHk} 
Assume $(D, C) \in \maD^{\ell, j}(B_r^m(0) \times [0, r); \CC^N)$
  satisfies an $\cdotH{\ell+1}(B_r^m(0) \times [0, r))^N$-regularity
    estimate, then it satisfies an $H^{\ell+1}$-regularity estimate
    for all $\ell\geq j+1$ ($\ell \geq j$ if $C$ is a variational
    boundary condition).  This result extends to uniform conditions.
\end{lemma}

\begin{proof} 
Denote $M = B_r^m(0) \times [0, r)$ for the simplicity of the
  notation.  Assume first that $k\geq 2$. Let $w \in \Gamma(M; E)$ be
  smooth with compact support in $M$ such that $D w \in H^{k-1}(M; E)$
  and $C w\in H^{k-j+1/2}(\pa M; E)$. We need to show that $w \in
  H^{k+1}(M; E)$. We have
\begin{align*}
  \|w\|_{H^{k+1}(M; E)}\leq & \ |w|_{H^{k+1}(M; E)} + \|w\|_{H^{k}(M;
    E)}\\
 \leq & \ c_{SL} \left( |D w|_{H^{k-1}} + |C w|_{H^{k-j+1/2}(\pa M;
   E)} \right ) + \| w \|_{H^{k}(M; E)} \\
 \leq & \ c \left(\|D w\|_{H^{k-1}(M; E)} + \| w \|_{H^{k}(M; E)} +
 \|C w\|_{H^{k-j+1/2}(\pa M; E)} \right ).
\end{align*}
This result is extended to $w \in H^{k}(M; E)$ by a continuity and
density argument using mollifying functions. The proof for $k = j =1$
and variational boundary conditions is similar.
\end{proof}

\begin{remark}\label{rem.usual.SL}
Let us assume that $r = \infty$ and consider the map
\begin{equation}\label{eq.def.op.bvp}
  (D, C) \colon \cdotH{k+1}(\RR^{m+1}_+)^N \to
  \cdotH{k-1}(\RR^{m+1}_+)^N \oplus \cdotH{k-j+1/2}(\RR^{m})^N,
\end{equation}
given by $(D, C)(u) = (Du, Cu)$. Then $(D, C)$ satisfies an
$\cdotH{\ell+1}(\RR^{m+1}_+)^N$--regularity estimate if, and only if,
$(D, C)$ is injective with closed range for all $j+1\leq k\leq
\ell$. Equivalently, we have that $(D,C)$ is injective and its minimal
reduced module $\gamma(D, C)> 0$, see Subsection~\ref{ssec.alt}. (In
this case, the number $\gamma(D, C)$ is the least $c_{SL}$ in
Definition~\ref{def.SL.reg}.)  In case $D$ and $C$ have constant
coefficients, this is similar to the Shapiro-Lopatinski condition.
\end{remark}

\subsection{A global Shapiro-Lopatinski regularity condition}
Let us assume first that we are on a Euclidean space and that $D =
\sum_{|\alpha|\leq 2} a_{\alpha} \pa^\alpha$, with $a_{\alpha}$
smooth, matrix valued functions. Recall that $j$ is the order of the
boundary conditions and that we assume for simplicity, that only one
of the vector bundles $F_0$ and $F_1$ (from $E \vert_{\pa M} = F_0
\oplus F_1$) is non-zero. Thus, according to our notational
convention, if $j = 0$, $C= C_{0}$ is a smooth, matrix valued function
on $B_r^m(0)$ and, if $j = 1$, $C = C_1 = C_{10} + C_{11} \pa_{\nu}$,
with $C_{10} = \sum_{|\alpha|\leq 1} c_{\alpha} \pa^\alpha$ a
\emph{first order} differential operator on $B_r^m(0)$ and $C_{11}$ a
smooth, matrix valued function on $B_r^m(0)$. Here $\pa_{\nu} = -
\pa_n$ is the outward pointing normal derivative, where $\pa_n$ is the
partial derivative with respect to the last variable. We denote by
$(D^{(0)}, C^{(0)})$ the \emph{principal part} of $(D, C)$ with
\emph{coefficients frozen at} $0$, that is,
\begin{align}\label{eq.def.principal.part}
  D^{(0)} & \define \ \sum_{|\alpha| = 2} a_{\alpha}(0) \pa^\alpha\\
  C_0^{(0)} & \define C_0(0), \ \mbox{ if } j = 0, \\
  C_1^{(0)} & \define \ \ \sum_{|\alpha| = 1} c_{\alpha}(0) \pa^\alpha
  + C_0(0) \pa_{\nu}, \ \mbox{ if } j = 1.
\end{align}
Thus $(D^{(0)}, C^{(0)})$ is a homogeneous, constant coefficient
boundary value problem on $\RR^{m+1}_+$. In case both $F_0$ and $F_1$
are non-zero, we let $C^{(0)} = (C_0^{(0)}, C_1^{(0)})$.

Motivated by Remark \ref{rem.usual.SL}, we introduce the following
definition.

\begin{definition} \label{def.SL.reg0}
We shall say that $(D, C) \in \maD^{\ell, j}(B_r^m(0) \times [0, r);
  \CC^N)$ satisfies the \emph{$H^{\ell+1}$-Shapiro-Lopatinski
    regularity condition at $0$} if $(D^{(0)}, C^{(0)})$ satisfies an
  $\cdotH{\ell+1}(\RR^{m+1}_+)^N$-regularity estimate.
\end{definition}

Let us turn now to the case of a manifold with boundary. As
  noticed already, all the needed definitions and concepts
  (homogeneous Sobolev spaces $\cdotH{\ell + 1}(V; E)$ and
  $\cdotH{\ell + 1}(V_+; E)$, regularity conditions, ... ) extend to a
  vector space (respectively, half-vector space) endowed with a
  metric. For instance, we define the principal part with coefficients
  frozen at a boundary point as follows.

\begin{notation}\label{not.princ}
Let $T_x^+M$ be the half-space of $T_xM$ corresponding to the inward
pointing vectors at $x \in \pa M$.  Let $(D_x, C_x)$ be the induced
operator (defined only a neighborhood of 0) on $T_x^+M$.  The
\emph{principal part} $(D^{(0)}_x, C^{(0)}_x)$ of $(D_x, C_x)$ with
\emph{coefficients frozen at} $x$ will then be a matrix of constant
coefficient differential operators on $T_x^+M$.
\end{notation}

Most importantly, the above definition (Definition~\ref{def.SL.reg0})
generalizes to $(D, C) \in \maD^{\ell, j}(M; E)$ and any point $x \in
\pa M$.

\begin{definition} \label{def.SL.reg.x}
We shall say that $(D, C) \in \maD^{\ell, j}(M; E)$ satisfies the
\emph{$H^{\ell+1}$-Shapiro-Lopatinski regularity condition at $x \in
  \pa M$} if $(D^{(0)}_x, C^{(0)}_x)$ satisfies the
$H^{\ell+1}$-Shapiro-Lopatinski regularity condition at 0 on $T_x^+M$.
\end{definition}

The above condition is closely related to the condition of
``regularity upon freezing the coefficients'' introduced in
\cite[Equation (11.30)]{Taylor1} and used, for instance, in
\cite{elasticity}. We are ready now to globalize the
Shapiro-Lopatinski regularity condition.
 
\begin{definition} \label{def.SL.reg.M}
We shall say that $(D, C) \in \maD^{\ell, j}(M; E)$ satisfies a
\emph{uniform $H^{\ell+1}$-Shapiro-Lopatinski regularity condition (at
  $\pa M$)} if it satisfies the $H^{\ell+1}$-Shapiro-Lopatinski
regularity condition at $x$ for each $x \in \pa M$ and the constant
$c_{SL}$ of Definition~\ref{def.SL.reg} can be chosen independently of
$x \in \pa M$.
\end{definition}

In particular, the constant $c_{SL}$ depends and scales with the
  metric; see Equation \eqref{eq.conf.inv} for the precise dependence
  on the metric. We thus see that $(D, C) \in \maD^{\ell, j}(M; E)$
satisfies the uniform, $H^{\ell+1}$-Shapiro-Lopatinski regularity
condition (at $\pa M$) if there exists $c_{SL} > 0$ such that, for all
$x \in \pa M$ and all $1 \leq k \leq \ell$, we have
\begin{equation}\label{eq.def.SL}
  | w |_{H^{k+1}(T_x^+M)} \leq c_{SL} \left(|D^{(0)}_x
  w|_{H^{k-1}(T_x^+M)} + |C^{(0)}_x w|_{H^{k-j+1/2}(T_x \pa
    M)}\right).
\end{equation} 
 
We now apply these notions to a manifold $M$ with boundary and bounded
geometry.  For any $x \in M$, we denote by $(D_x, C_x)$ (or simply by
$D_x$, if $x \notin \pa M$) the operators (respectively,
operator) on $B_r^m(0) \times [0, r)$ (respectively, on $B_r^m(0)$)
  induced by $(D, C)$ (respectively, by $D$) in Fermi
  coordinates around $x$. We let $\maF_b = \{ (D_x, C_x) \, \vert \ x
  \in \pa M\}$ and $\maF_i = \{ D_x \, \vert \ \dist(x, \pa M)\geq
  r\}$, as in \eqref{eq.def.F}. We have the following theorem.

\begin{theorem}\label{thm.SL} 
Assume that $M$ is a manifold with boundary and bounded geometry and
that $E \to M$ is a vector bundle with bounded geometry.  Let $(D, C)
\in \maD^{\ell+1, j}(M; E)$.  The following are equivalent.
\begin{enumerate}[(i)]
 \item $(D, C)$ satisfies an $H^{\ell+1}$-regularity estimate on $M$.
 \item The family $\maF_b \cup \maF_i = \{ (D_x, C_x) \, \vert \ x \in
   \pa M\} \cup \{ D_x \, \vert \ \dist(x, \pa M)\geq r\}$ satisfies a
   uniform $H^{\ell+1}$-regularity estimate.
 \item $D$ is uniformly elliptic on $M$ and $(D, C)$ satisfies a
   uniform $H^{\ell+1}$-Shapiro-Lopatinski regularity condition (at
   $\pa M$).
 \item $D$ is uniformly elliptic and $(D, C)$ satisfies a uniform
   $H^2$-Shapiro-Lopatinski regularity condition (at $\pa M$).
\end{enumerate}
If $(D, C)$ are obtained from a variational formulation, then the
above conditions are equivalent also to
\begin{enumerate}
 \item[(v)] $D$ is uniformly elliptic and $(D, C)$ satisfies a
   uniform $H^1\!$-Shapiro-Lopatinski regularity condition (at $\pa
   M$).
\end{enumerate}
\end{theorem}

Note that usually we assume $(D, C) \in \maD^{\ell, j}(M; E)$, whereas
here we assume $(D, C) \in \maD^{\ell+1, j}(M; E)$. This is needed for
the proof of (iii) $\Rightarrow$ (i).

\begin{proof} 
The implication (i) $\Rightarrow$ (ii) follows from the definitions by
localization. (This is the converse of Theorem~\ref{thm.reg0}.)
 
To obtain (ii) $\Rightarrow$ (iii) for each $x \in \pa M$, we consider
$\xi$ with compact support on $T_x^+M$ and $\xi_\epsilon(v) =
\xi(\epsilon^{-1}v)$. Using a chart around $x$ with Fermi coordinates
as in Subsection~\ref{ssec.cov} and using the corresponding bounds of
the transition function, the uniform $H^{\ell+1}$-regularity estimate
for each $\epsilon > 0$ from Definition~\ref{def.loc.reg} implies
\begin{align*}
  \epsilon^{k+1-n/2}& \| \xi_\epsilon \|_{H^{k+1}(T_x^+M; E)}\leq
  \epsilon^{k+1-n/2} \bar{c}_R \left( \| \xi_\epsilon
  \|_{H^{k}(T_x^+M; E)} \right . \\
  & \ \left . + \epsilon^{k+1-n/2} \|D_x \xi_\epsilon
  \|_{H^{k-1}(T_x^+M; E)} + \|C_x \xi_\epsilon\|_{H^{k-j+1/2}(\pa
    T_x^+ M; E)}\right)
\end{align*} 
for some $\bar{c}_R$ independent on $x$.  Passing to the limit as
$\epsilon \to 0$ and using Lemma \ref{lemma.classical}(ii) we obtain
\begin{equation*}
  | \xi |_{H^{k+1}(T_x^+M; E)}\ \leq\ \bar{c}_R \left(|D_x^{(0)} \xi
  |_{H^{k-1}(T_x^+M; E)} + |C_x^{(0)} \xi |_{H^{k-j+1/2}(\pa T_x^+ M;
  |E)}\right).
\end{equation*}
This gives right away that $(D, C)$ satisfies a uniform,
$H^{k+1}$-Shapiro-Lopatinski regularity condition (at $\pa M$).  The
same argument combined with a Fourier transform at an arbitrary
interior point (and $ 0< r < r_{FC}$ arbitrary) and a perturbation
argument \cite{Hormander3, Taylor1} gives that $D$ is uniformly
elliptic on $M \smallsetminus \pa M$. Hence $D$ is uniformly elliptic
on $M$.
 
The implications (iii) $\Rightarrow$ (iv) $\Rightarrow$ (v) are
trivial. We have that, in fact, they are equivalent. This is seen in
the same way in which one proves higher regularity for boundary value
problems using divided differences. See any textbook, in particular
\cite{EvansBook, JostBook, LionsMagenes1, Taylor1}.
 
To complete the proof, it is enough then to show that (iii)
$\Rightarrow$ (i).  We want to estimate $\|u\|_{H^{k+1}}$ in terms of
$\|Du\|_{H^{k-1}}$, $\|Cu\|_{H^{k-j-1/2}}$ and $\|u\|_{H^{k}}$.  This
is done using an $r$-partition of unity $\phi_\gamma$ as in
Definition~\ref{def_part} and following then almost word for word the
proof of Theorem~\ref{thm.reg0}, but choosing $r >$ small enough. Let
us use the notation of the proof of that theorem. Then the only
difference in the estimate of the proof is that we need to replace
$P_\gamma$ and $C_\gamma$ with their principal parts $P_\gamma^{(0)}$
and $C_\gamma^{(0)}$ and with coefficients frozen at $p_\gamma$. We
note that the family $(P_\gamma^{(0)}, C_\gamma^{(0)})$ satisfies a
uniform $H^{\ell+1}$-regularity estimate, in view of Lemma
\ref{lemma.dotHk}. The lower order terms can be absorbed into the
weaker norm $\|u\|_{H^k}$. We then use the fact that $\|(P -
P_\gamma)u\|_{H^{k-1}}\leq C(r) \|u\|_{H^{k+1}}$, for $u$ with support
in the ball of radius $r$ centered at $\gamma$ and with $C(r)$
independent of $\gamma$ and with $C(r) \to 0$ as $r \to 0$. We have
$C(r) \to 0$ when $r \to 0$ since $P$ has coefficients in $W^{\ell+1,
  \infty}$.  To obtain regularity estimates in the interior (away from
the boundary), we use the uniform ellipticity of the operator.  For
more details, one can consult also Proposition 11.2 of \cite{Taylor1},
which is a similar result with a similar proof.
\end{proof}

In particular, the sequence (ii) $\Rightarrow$ (iii) $\Rightarrow$ (i)
gives a new proof of Theorem~\ref{thm.reg0}. We obtain the following
consequence.

\begin{corollary}
Let $(D, C) \in \maD^{\ell, j}(M; E)$.  If $(D, C)$ satisfies an order
$H^2$-regularity estimate on $M$, then $(D, C)$ satisfies an
$H^{\ell+1}$-regularity estimate on $M$.
\end{corollary}

It would be interesting to investigate the relation between the
results of this paper and those of Karsten Bohlen \cite{KarstenCR,
  KarstenBMB}.

\subsection{A uniform Agmon condition} In view of the results on
the uniform Shapiro-Lopatinski conditions and of the usefulness
  of positivity apparent in the next section, we now consider the
coercivity of our operators, in the same spirit as the uniform
Shapiro-Lopatinski conditions. The notation and the approach are very
similar.  We keep the notation of Section~\ref{sec.variational}. In
particular, $\tilde P$ will be a second order differential operator in
divergence form with associated boundary conditions $C = (C_0, C_1)$,
as in Subsection \ref{ssec.hr}. \emph{From now on, $M$ will be a
  manifold with boundary and bounded geometry.}

Recall the following standard terminology.

\begin{definition}\label{def.s.coercive}
A sesquilinear form $a$ on a hermitian vector bundle $V \to X$ is
\emph{called strongly} coercive (or \emph{strictly positive}) if there
is some $c>0$ such that $\Re a(\xi,\xi)\geq c|\xi|^2$ for all $x\in X$
and $\xi \in V_x$. If the sesquilinear form $a$ on $T^*M \otimes E$
used to define $P$ is strongly coercive, then $P$ is said to satisfy
the \emph{strong Legendre condition}.
\end{definition}

Let $V \subset H$ be a continuous inclusion of Hilbert spaces (with
non-closed image, in general). Let $V^*$ be the complex conjugate of
$V$ with pairing $V \times V^* \to \CC$ restricting to the scalar
product of $H$ on $V \times H$. Recall that an operator $T \colon V
\to V^*$ is \emph{coercive} on $V$ if it satisfies the G{\aa}rding
inequality, that is, if there exist $\gamma > 0$ and $R \in \RR$ such
that, for all $u\in V$,
\begin{equation}\label{eq.Gaarding}
  \Re \<T u, u\> \geq \gamma \|u\|_{V}^2 - R \|u\|^2_{H}.
\end{equation}
Then $T + \lambda$ is strongly coercive for $\Re(\lambda) > R$ (see
Definition~\ref{def.s.coercive}), and hence it satisfies the
conditions of the Lax-Milgram lemma.  Therefore, it satisfies
regularity in view of the results of the next section. Coercive
operators on \emph{bounded} domains were characterized by Agmon in
\cite{AgmonCoercive} as strongly elliptic operators satisfying
suitable conditions at the boundary (which we shall call the ``Agmon
condition.''). We shall need a \emph{uniform} version of this
condition, to account for the non-compactness of the boundary.

Let now $\tilde P \tilde V \to V^*$ be as in Equation
\eqref{eq.def.tildeP} (so it is associated to the sesquilinear form
$B$ and has principal symbol $a$). Let $C$ be the boundary conditions
associated to $\tilde P$. That is, $C = (e_0, e_1 \pa_{\nu}^{a} +
Q)$. We let $e_{0x}$ and $e_{1x}$ be the values at $x$ of the
endomorphisms $e_0$ and $e_1$. Similarly, let $Q_x^{(0)}$ be the
principal part of $Q$ with coefficients frozen at $0$, where
  $Q_x$ is regarded as a first order differential operator (so
  $Q_x^{(0)}= 0$ if $Q$ is of order zero. Let $P^{(0)}_x$ be the
principal part of the operator $P$ and $C_x^{(0)} = (e_{0x}, e_{1x}
\partial_\nu^a + Q_x^{(0)})$ be the principal part of the boundary
conditions $C$ with coefficients frozen at some $x \in \pa M$, as in
Notation \ref{not.princ}. Let $B_x^{(0)}$ be the associated Dirichlet
bilinear form to $P^{(0)}_x$ equipped with the above boundary
conditions (again with coefficients frozen at $x$), that is
\begin{equation}\label{def.Bx0}
  B_x^{(0)}(u, v) \define \int_{T_x^+} a_x^{(0)}(du, dv) + \int_{T_x
    \pa M} ( Q_x^{(0)} u, v).
\end{equation}
This defines a continuous sesquilinear form on
\begin{equation*}
  V_{x} \define \{ u \in H^1(T_x^+M; E_x) \, \vert \ u\vert_{T_x \pa
    M} \in (F_{1})_x \}.
\end{equation*}
The associated operator in divergence form $\tilde P^{(0)}_x$ will be
called \emph{the principal part of $\tilde P$ with coefficients frozen
  at $x$.}

\begin{definition}
We say that $\tilde P$ satisfies the \emph{uniform Agmon condition (on
  $\pa M$)} if it is uniformly strongly elliptic and if there exists
$C > 0$ such
\begin{equation*}
  \<\tilde P^{(0)}_x u, u\> = B_x^{(0)}(u, u) \geq C |u|_{H^1}^2,
\end{equation*}
for all $x \in \pa M$ and all $u \in \CIc(T_x^+M)$ with $e_{0x} u = 0$
on the boundary of $T_x^+M$.
\end{definition}

We have then the following result that is proved, \emph{mutatis
  mutandis}, as the regularity result of Theorem~\ref{thm.SL}.

\begin{theorem}\label{thm.Agmon}
Let $M$ be a manifold with boundary and bounded geometry, $E \to M$ be
a vector bundle with bounded geometry, and $\tilde P$ be a second
order differential operator in divergence form, as above. We assume
$\tilde P$ has coefficients in $W^{1,\infty}$.  We have that $\tilde
P$ is coercive on $V \define \{ u \in H^1(M; E) \, | \ e_0 u = 0
\mbox{ on } \pa M\} $ if, and only if, $\tilde P$ is uniformly
strongly elliptic and it satisfies the uniform Agmon condition on $\pa
M$.
\end{theorem}

The proof is essentially the same as that of Theorem~\ref{thm.SL},
more precisely, of the equivalence (i) $\Leftrightarrow$ (iii). We
need at least $W^{1,\infty}$ to make the partition of unity argument
work.

\begin{remark}\label{rem.max.reg}
As is well-known, coercivity estimates lead to solutions of evolution
equations \cite{AmannParab, AmannCauchy, LionsMagenes1, PazyBook}. Let
$H_0^1(M; E) \subset V \subset H^1(M; E)$ be the space defining our
variational boundary value problem, see \ref{assume}, Equation
\eqref{eq.def.VnW}, and Remark \ref{rem.higher.order}. Let $V^*$ be
the complex conjugate dual of $V$, as before, $\maW := L^2(0, T; V)$,
$T>0$, so that $\maW^* = L^2(0, T; V^*)$. Assume $P \colon V \to V^*$
is coercive (i.e. it satisfies Equation \eqref{eq.Gaarding}). Then
Theorem 4.1 of \cite[Section 3.4.4]{LionsMagenes1} states that, for
any $f \in \maW^*$, there exists a unique $w \in \maW \cap C([0, T],
L^2(M; E))$ such that $\partial_t w(t) - Pw(t) = f(t)$ and $w(0) =
0$. Moreover, we also have $w \in H^1(0, T; V^*)$.
\end{remark}

This leads to the following result.

\begin{theorem}\label{thm.max.reg}
Let us assume that $\tilde P$ is as in Remark \ref{rem.max.reg}, whose
notation we continue to use, and that $\tilde P$ satisfies the uniform
Agmon condition. Then, for any $f \in \maW^*$, there is a unique $w
\in \maW \cap C([0, T], L^2(M; E))$ such that $\partial_t w(t) - Pw(t)
= f(t)$ and $w(0) = 0$. Moreover, we also have $w \in H^1(0, T; V^*)$.
\end{theorem}

For manifolds with bounded geometry (no boundary), this result was
proved in \cite{MN1}. The result in \cite{MN1} was generalized to
higher order equations in \cite{AmannCauchy}. For the particular case
mixed (Dirichlet/Neumann) boundary conditions and scalar equations,
this result was proved in \cite{AmannParab}.

\subsection{Conformal invariance}  Both the uniform
Shapiro-Lopatinski regularity condition and the uniform Agmon
condition are conformally invariant in an obvious sense that we make
explicit in this subsection. Let $\rho > 0$ be a smooth function on
$M$ such that $\rho^{-1}d\rho$ is in $W^{\infty, \infty}$ (a function
$\rho$ with these properties will be called an \emph{admissible}
weight). Let $(P, C) = (P, C_0, C_1)$ be a boundary value problem. (We
no longer assume that $C$ has constant order on the boundary). Recall
that the metric on $M$ is denoted $g$, and consider $g':= \rho^{-2}
g$, $P' \define \rho^2 P$, $C'_0 \define C_0$, $C'_1 \define \rho C_1$, and $C' \define (C'_0, C'_1)$. All differential operators will act on the same vector bundle $E$, whose metric we do \emph{not} change.

\begin{proposition}\label{prop.conf.inv}
Assume that $(P, C)$ has coefficients in $W^{\ell+1, \infty}$
(i.e. $(P, C) \in \maD^{\ell+1,j}(M; E)$ if $C$ is of constant order
$j$). We have that $(P, C)$ satisfies a uniform
$H^{\ell+1}$-Shapiro-Lopatinski regularity condition (with respect to
the metric $g$) if, and only if, $(P', C')$ satisfies a uniform
$H^{\ell+1}$-Shapiro-Lopatinski regularity condition (with respect to
the metric $g'$). The same statement remains true if we replace ``a
uniform $H^{\ell+1}$-Shapiro-Lopatinski regularity condition'' with
``a uniform Agmon condition.''
\end{proposition}

\begin{proof}
This follows directly from definitions, by taking into account how the
homogeneous Sobolev space norms change under the change of
metric. More precisely, the semi-norm $| \cdot |'_{H^s} = | \cdot
|'_{H^s(T_xM)}$ associated to the metric $g' \define \rho^{-2}g =
\rho(x)^{-2}g$ on $T_xM$ is related to the original semi-norm $| \cdot
|_{H^s} = | \cdot |_{H^s(T_xM)}$ associated to the metric $g$ on
$T_xM$ by the relation 
\begin{equation}\label{eq.conf.inv}
  |v|'_{H^s} = \rho(x)^{s + m/2} |v|_{H^s},
\end{equation}
where $m$ is the dimension of $M$. Taking into account this equation
and, assuming, for simplicity that we have constant order $j$ at the
boundary, we have the following (where $\rho = \rho(x)$)
\begin{align*}
    |w|_{H^{k+1}(T_x^+M)} \leq &c_{SL} \left(|P^{(0)}_x
  w|_{H^{k-1}(T_x^+M)} + |C^{(0)}_x w|_{H^{k-j+1/2}(T_x \pa M)}\right)
  \\
   \Leftrightarrow  \rho^{k+1+ m/2}|w|_{H^{k+1}(T_x^+M)} 
   \leq&
   \rho^{k+1+ m/2} c_{SL} \left(|P^{(0)}_x w|_{H^{k-1}(T_x^+M)} \right. \\
   & \left.+    |C^{(0)}_x w|_{H^{k-j+1/2}(T_x \pa M)}\right) \\
   \Leftrightarrow  |w|'_{H^{k+1}(T_x^+M)} \leq &c_{SL} \left(
   \rho^{k-1+ m/2} |\rho^2 P^{(0)}_x w|_{H^{k-1}(T_x^+M)}\right.\\
   & \left.+ \rho^{k-j
     + 1/2 + (m-1) /2} | \rho^j C^{(0)}_x w|_{H^{k-j+1/2}(T_x \pa
     M)}\right) \\
    \Leftrightarrow  |w|'_{H^{k+1}(T_x^+M)} \leq &c_{SL}\! \left(
    |(P')^{(0)}_x w|'_{H^{k-1}(T_x^+M)}\! +\! |(C')^{(0)}_x
    w|'_{H^{k-j+1/2}(T_x \pa M)}\right)\!.
\end{align*}
This completes the proof for the case of Shapiro-Lopatinski regularity
condition, in view of the definition of the uniform
$\cdotH{\ell+1}$-Shapiro-Lopatinski regularity condition, Definition
\ref{def.SL.reg.M}.

The proof for the uniform Agmon condition is completely similar (only
shorter), once one notices that the ``full'' operator $\tilde P'$
associated to $(P', C')$ (and the associated bilinear form) scales in
the right way, that is $\tilde P' = \rho^2 \tilde P$.
\end{proof}

Note that for the above proof we did not need that $\rho$ be an
admissible weight. We continue to use the notation $(P', C')$
introduced right before Proposition \ref{prop.conf.inv}. We obtain the
following consequence. The regularity estimates and the coercivity are
stable under conformal changes of metric with bounded, admissible
weights. More precisely, we have the following theorem.

\begin{theorem}
Assume that $(P, C)$ has coefficients in $W^{\ell+1, \infty}$ and that
$\rho$ is an admissible weight. Then $(P, C)$ satisfies an
$H^{\ell+1}$-regularity estimate on $M$ (with respect to the metric
$g$) if, and only if, $(P', C')$ satisfies an $H^{\ell+1}$-regularity
estimate on $M$ (with respect to the metric $g'$). The same statement
remains true for coercivity.
\end{theorem}

\begin{proof}
Let us notice that $P$ is uniformly elliptic (respectively, uniformly
strongly elliptic) with respect to the metric $g$ if, and only if,
$P'$ satisfies the same property for the metric $g'$. Also, the fact
that the weight is admissible and bounded guarantees that $P'$ is in
divergence form with bounded coefficients. Then the result follows by
combining Proposition \ref{prop.conf.inv} with Theorem~\ref{thm.SL}
(for the regularity part), respectively with Theorem~\ref{thm.Agmon}
for the coercivity part.
\end{proof}

\section{Coercivity, Legendre condition, and regularity}
\label{sec.coercive}

In this section we use coercivity (or positivity) to obtain regularity
results. As an application, we study mixed Dirichlet/Robin boundary
conditions for operators satisfying the strong Legendre condition. We
continue to assume that $M$ is a manifold with boundary and bounded
geometry and that $E \to M$ is a vector bundle with bounded geometry.

\subsection{Well-posedness in energy spaces implies regularity}
Often the regularity conditions (including the Shapiro-Lopatinski ones
discussed below) are obtained from the invertibility of the given
operator. This is the case also in the bounded geometry setting, owing
to Proposition \ref{prop.alternative}. The results of this subsection
will be used to deal with the Neumann and Robin boundary conditions.
The approach below is based on the so called ``Nirenberg trick'' (see
also \cite{elasticity} and the references therein).

\begin{lemma} \label{lemma.1param}
Let $X_j$ be as in Lemma \ref{lemma.prop.exists} and $Y = X_j$, for
some $j > 1$ fixed.

(i) There exists a one parameter group of diffeomorphisms $\phi_t
\colon M \to M$, $t\in \mathbb R$, that integrates $Y$, that is
$\frac{d}{dt} f(\phi_t(x))_{t=0} = (Y f)(x)$, for any $x \in M$ and
any smooth function $f \colon M \to \CC$.

(ii) Let $\pi\colon E\to M$ be a vector bundle. Then there exists a
one parameter group of diffeomorphisms $\psi_t \colon E \to E$, $t \in
\RR$, with $\pi \circ \psi_t = \phi_t \circ \pi$, that integrates
$\nabla_{Y}$, that is $\frac{d}{dt} (\psi_{-t}\circ \xi)|_{t=0} =
\nabla_Y \xi$, for any smooth section $\xi \colon M \to E$.
\end{lemma}

\begin{proof}
\noindent (i) Let $\hat{M}$ be as in
Definition~\eqref{def_bdd_geo}. Let $\hat{Y}\in W^{\infty,
  \infty}(\hat{M}, T\hat{M})$ be an extension of $Y$. Then $\hat{Y}$
is a bounded vector field on a complete manifold. Hence, by
\cite[Sec. 3.9]{AF} it admits a global flow, i.e. a smooth solution
$\phi\colon \mathbb R \times \hat{M}\to \hat{M}$ of
\begin{align*}
   \frac{d}{dt}\phi(t,p)= \hat{Y}(p),\quad \phi(0,p)=p
\end{align*}
such that $\phi_t\define \phi (t,.)$ is a one-parameter family of
diffeomorphisms of $\hat{M}$. Since $\hat{Y}$ is tangent to $\partial
M\subset \hat{M}$, $\phi(t, \partial M)=\partial M$.  Since $\partial
M$ divides $\hat{M}$ into two parts, $\phi_t$ restricts to
diffeomorphisms of $M$.

\noindent (ii) We choose a connection $\nabla^E$ on $E$. Let
$\psi_t\colon E\to E$ be defined by $e\mapsto e(t)$ where $e(t)$ is
the solution of $\nabla^E_{Y(p)=\partial_t \phi_t(p)} e(t)=0$ with
$e(0)=e$. By the standard properties of parallel transport,
respectively of the underlying linear ordinary differential equation,
we have the global existence and uniqueness of the solution and the
claimed properties.
\end{proof}

We then obtain the following abstract regularity result. Let $X_j$ be
as in Lemmas \ref{lemma.prop.exists} and \ref{lemma.1param}. We can
assume that $X_1$ is a unit vector field normal at the boundary.
Recall that $V \define H^1(M; E) \cap \{ u \vert_{\pa M} \in \Gamma(\pa M;
F_1)\}$.

\begin{theorem} \label{thm.wp}
Let $\tilde P$ be a second order differential operator in divergence
form with associated form $a$ and $W^{\ell, \infty}$-coefficients.
Let $\zeta \define a(X_1, X_1)$ and let us assume that $\zeta$ is
invertible and $\zeta^{-1}$ bounded.  Also, let us assume that $\tilde
P \colon V \to V^*$ is a continuous bijection (i.e., an isomorphism).
Then $P$ satisfies an $H^{\ell+1}$-regularity estimate on $M$.
\end{theorem}

\begin{proof}
The proof is classic, except maybe the fact that we have a slightly
weaker assumption on the coefficients; typically one requires
$C^{k+1}$-coefficients in textbooks. We include, nevertheless, a very
brief outline of the proof. Recall that the proof is done by
induction, with the general step the same as the first step (going
from well-posedness to $H^2$-regularity). Let us assume then that
$\ell = 1$.

Let $F \in j_0(L^2(M; E) \oplus H^{1/2}(M; F_1))$ and let $\tilde P u
= F$, with $u \in V$. For simplicity, let us assume $E$ is
one-dimensional. In general, we just replace $X_j$ with
$\nabla_{X_j}$. We want to show that $u \in H^2(M)$, with continuous
dependence on $F$. To this end, in view of Proposition
\ref{prop.alternative}, it is enough to check that $X_i X_j u \in
L^2(M)$, with continuous dependence on $F$, since we already know that
$u \in V \subset H^1(M)$, with continuous dependence on $F$.

In particular, $X_j u \in L^2(M)$. Nirenberg's trick is to give
conditions on $X_j$ such that we can formally apply $X_j$ to the
equation $\tilde P u = F$ to obtain that $X_j u \in V$ is the unique
  solution of
\begin{equation}\label{eq.Nirenberg}
   \tilde P( X_j u ) = [\tilde P, X_j](u) + X_j (F) \in V^*.
\end{equation}
This is possible whenever $[\tilde P, X_j] \colon V \to V^*$ continuously
and $X_j$ generates a continuous parameter semi-group on $V$ (see
\cite{BHN} for a general approach and more details).  These conditions
are satisfied since $\tilde P$ has coefficients in $W^{1, \infty}$ and
if $j > 1$, since then $X_j$ is tangent to the boundary and we can
invoke Lemma \ref{lemma.1param}. This argument gives that $X_i X_j u
\in L^2(M)$ if \emph{at least} one of the $i$ and $j$ is $> 1$.

It remains to prove that $X_1^2 u \in L^2(M)$. This is proved using
the equation
\begin{equation*}
   \zeta X_1^2 u = Pu - \sum_{i+j>2} c_{ij} X_i X_j u \in L^2(M),
\end{equation*}
since we can choose $c_{ij} \in L^\infty(M; \End(E))$ and $Pu
  \in L^2(M; E)$, by assumption.  (Recall that we assumed that $X_1$
is a unit vector everywhere  on the boundary.)
\end{proof}

\begin{remark}\label{rem.in.general}
Except the above theorem, it is very likely that most of the results
obtained so far extend to higher order equations, but we have not
checked all the details. The above theorem will require, however, some
additional ideas in order to extend it to the setting of Remark
\ref{rem.higher.order}.
\end{remark}

\subsection{Robin vs Shapiro-Lopatinski}
Let us discuss, as an example, Robin (and hence also Neumann) boundary
conditions from the perspective of the Shapiro-Lopatinski conditions.
We do that now in the case of a model problem on the half-space
$\RR^n_{+} = \RR^{n-1} \times [0, \infty)$. The sesquilinear form $a$
  of Section~\ref{sec.variational} (see Assumption \ref{assume}) is
  now simply a sesquilinear form on $\CC^{nN} = \RR^n \otimes
  \CC^N$. We allow now $F_0$ and $F_1$ to be both non-trivial, which
  amounts to a decomposition $N = N_0 + N_1$, $N_j \geq 0$. We shall
  need also the operator $Q$, which is now a $N_1 \times N_1$ matrix
  of constant coefficients differential operators on $\RR^{n-1}$
  acting on the last $N_1$ components of $\CC^N$. The bilinear form
  that we consider is then
\begin{equation}\label{eq.new.B}
  B(u, v) \define \int_{\RR^n_{+}} \big [ a(du, dv) + c(u, v) \big ]
  dx + \int_{\RR^{n-1}} (Q u, v) dx',
\end{equation}
where $c$ is a scalar and $x = (x', x_n)$. This is a particular case
of the form considered in Equation \eqref{eq.def.B}. Recall the
definition of a strongly coercive form, Definition
\ref{def.s.coercive}. Note that strong coercivity implies uniform
strong ellipticity as in Definition~\ref{def.u.s.e}.

\begin{lemma}\label{lemma.ccSL}
Let us assume that $Q + Q^*$ is of order zero and that $a$ is strongly
coercive. Then for $c$ (of Equation \ref{eq.new.B}) large enough,
there is a $\gamma > 0$ such that
\begin{equation}
   B(u, u) \geq \gamma \|u\|_{H^1(\RR^n_{+})}^2
\end{equation}
for all $u \in H^1(\RR^n_{+} )$.
\end{lemma}

\begin{proof}
We identify $\RR^{n-1}$ with the boundary of the half-space
$\RR^n_{+}$, as usual. Let $c_Q> 0$ be a bound for the norm of the
matrix $\frac12(Q + Q^*)$.  We have, using first the definition
\begin{align*}
  B(u, u) \define & \ \int_{\RR^n_{+}} \big [ a(du, du) + c(u, u) \big
  ] dx + \int_{\RR^{n-1}} (Q u, u) dx'\\
  \geq & \ c_a |u|_{H^1(\RR^{n}_+)}^2 + c \|u\|_{L^2(\RR^{n}_+)}^2 -
  c_Q \|u\|_{L^2(\RR^{n-1})}\\
  \geq & \ \frac{c_a}{2} \|u\|_{H^1(\RR^{n}_+)}^2
\end{align*}
for $c$ large. The last statement is proved in the same way one proves
the trace inequality $\|u\|_{L^2(\RR^{n-1})} \le c_T
\|u\|_{H^1(\RR^{n}_+)}$, but using also Lebesgue’s dominated
convergence theorem.
\end{proof}

We obtain the following consequence.

\begin{corollary}\label{cor.RobinSL} 
Let $a$ and $Q$ be as in Lemma \ref{lemma.ccSL} and let $\tilde P$ be
the differential operator in divergence form associated to the form
$B$ of Equation \eqref{eq.new.B}.  Then $\tilde P$ (or, $(P, e_0,
e_1\pa_{\nu}^{a} + Q)$) satisfies an $H^{\ell+1}$-regularity estimate
on $\RR^n_{+}$ for all $\ell$. In particular, $(P^{(0)}, e_0,
e_1\pa_{\nu}^{a} + Q^{(0)})$ satisfies the
$H^{\ell+1}$-Shapiro-Lopatinski regularity condition at $0$.
\end{corollary}

\begin{proof}
Let $V \define \{u \in H^1(\RR^n_{+}) \, | \ e_0 u = 0 \mbox{ on }
\RR^{n-1} \}$.  We then have that the restriction of $B$ to $V \times
V$ satisfies the assumptions of the Lax-Milgram lemma
\cite{TrudingerBook}, and hence $\tilde P \colon V \to V^*$ is an
isomorphism.  Theorem~\ref{thm.wp} then gives the first part of the
result. The last part follows from Theorem~\ref{thm.SL}.
\end{proof}

\subsection{Mixed Dirichlet/Robin boundary conditions}\label{ssec.mixed}
Let us turn now back to the study of mixed Dirichlet/Robin boundary
value problems on $M$. Let $M$ be a manifold with bounded geometry
with a decomposition $E\vert_{\pa M} = F_0 \oplus F_1$ as the direct
sum of two vector bundles with bounded geometry.  We consider the same
bilinear form $B$ as in Section~\ref{sec.variational}, see
Equation~\eqref{eq.def.B}, with the data defining it as in Assumptions
\eqref{assume}. In particular, $B$ is obtained from a form $a$ that
can be interpreted as the principal symbol of the associated operators
$P$ and $\tilde P$. If $a$ is strongly coercive, we say that $P$ (or
$\tilde P$) satisfies the \emph{strong Legendre condition.} (See
\cite{KohrPinteaWendland13} for a similar concept for Stokes-type
operators.)

The associated boundary conditions are then
\begin{equation}\label{eq.def.C}
  C_0 u = e_0 u\vert_{\pa M} \quad \mbox{ and } \quad C_1 u = e_1
  \pa_{\nu}^{a} u + Qu \vert_{\pa M}
\end{equation}
for some first order differential operator $Q$ acting on sections of
$F_1$. We let
\begin{equation*}
 \| C v \|_{k} \define \|C_0 v\|_{H^{k+1/2}(\pa M; F_0)} + \|C_1
 v\|_{H^{k-1/2}(\pa M; F_1)}.
\end{equation*}
We say that the boundary conditions $C$ are \emph{mixed
  Dirichlet/Robin boundary conditions.} We shall need the following
analog of Corollary \ref{cor.Dir}.

\begin{corollary}\label{cor.DR}
Let $S \subset \maD^{\ell+1, 0}(B^m_r(0) \times [0, r); E)$ be a
  \emph{bounded} family of boundary value problems on
$B^m_r(0) \times [0, r)\subset \RR^{m+1}$ equipped with the euclidean
metric, $r \leq \infty$. We assume that the family $S$ satisfies a
uniform strong Legendre condition and all the boundary conditions are
are Robin boundary conditions of the form $(e_0, e_1 \pa_{\nu}^{a} +
Q)$, with $Q + Q^*$ of order zero and bounded on $S$. Then the family
$S$ satisfies a uniform $H^{\ell+1}$-regularity estimate on
$B^m_{r'}(0) \times [0, r')$, $r'<r$.
\end{corollary}

\begin{proof}
Let $(D_n, C_n) \in S$ converge to $(D, C) \in \maD^{\ell, j}
(B^m_{r'}(0) \times [0, r'); E)$, with $C_n = (e_0, e_1
  \pa_{\nu}^{a_n} + Q_n)$, with $Q_n + Q_n^*$ of order zero. Then $D$
  satisfies the strong Legendre condition because the parameter $c_a$
  in the definition of the strong Legendre condition (the fact
that $a$ is strongly coercive) is assumed to stay away from 0 on $S$.
The limit of Robin boundary conditions is again a Robin boundary
condition and the condition that $Q + Q^*$ be scalar is also preserved
under limits.  Corollary~\ref{cor.RobinSL} then gives that $(D, C)$
satisfies an $H^{\ell+1}$-regularity estimate.  This allows us again
to use Proposition~\ref{prop.limit.reg} for the relatively compact
subset $N \define B^m_{r'}(0) \times [0, r')$ of $M \define B^m_{r}(0)
  \times [0, r)$ to obtain the result.
\end{proof}

We are ready now to prove the result stated in the Introduction,
Theorem~\ref{thm.intro.reg}. 

\begin{proof}
The proof of Theorem~\ref{thm.intro.reg} is the same as that of
Theorem~\ref{thm.reg.D}. Indeed Corollaries \ref{cor.DR} and
\ref{cor.unif.elliptic} show that the assumptions of Theorem
\ref{thm.reg0} are satisfied. That theorem immediately gives our
result.
\end{proof}

Notice that in the statement of the theorem, we have dropped the
condition that $B$ be of constant order (it is, nevertheless, of
\emph{locally constant} order).

Let us assume now that we have a partition of the boundary $\pa M =
\pa_D M \sqcup \pa_R M$ as a disjoint union of two open and closed
subsets and that $F_0 = E\vert_{\pa_D M}$ and $F_1 = E\vert_{\pa_R
  M}$. By combining Theorem~\ref{thm.wp} with the Poincar\'e
inequality \cite{AGN3}, we can prove following well-posedness result
for the mixed Dirichlet/Robin boundary value problem. Let $A \subset
\pa M$ (see \cite{AGN3} for details). Recall that we say that $(M, A)$
has finite width \cite{AGN1} if the distance to $A$ is bounded
uniformly on $M$ and $A$ intersects all connected components of $M$.

\begin{theorem}\label{thm.well-posedness1}
We use the same notation as in Theorem~\ref{thm.intro.reg}, in
particular, $M$ is a manifold with boundary and bounded geometry and
$\tilde P = (P, \pa_{\nu}^{a} u + Q)$ has coefficients in $W^{\ell,
  \infty}$ and satisfies the strong Legendre condition. Assume that
$F_0 = E\vert_{\pa_D M}$, that $F_1 = E\vert_{\pa_R M}$, that $Q + Q^*
\geq 0$, that there exists $\epsilon > 0$ and an open and closed
subset $\pa_{PR}M \subset \pa_{R}M$ such that $Q + Q^* \geq \epsilon$
on $\pa_{PR}M$, and, finally, that $(M, \pa_D M \cup \pa_{PR}M)$ has
finite width. Then the boundary value problem
\begin{align}\label{eq.mixed.Neumann}
  \left\{ 
  \begin{aligned} P u & = f \in H^{k-1}(M; E) && \text{in } M\\
  u & = h_D \in H^{k+1/2}(\pa_D M; E) && \text{on } \pa_D M\\
  \pa_{\nu}^{a} u + Qu & = h_R \in H^{k-1/2}(\pa_R M; E) && \text{on }
  \pa_{R} M ,
  \end{aligned} \right. 
\end{align}
has a unique solution $u \in H^{k+1}(M; E)$, $k\geq 0$, and this
solution depends continuously on the data, $0 \leq k \leq \ell$.
\end{theorem}

Note that the more general form of the Robin boundary conditions
considered in this paper (i.e. corresponding to a splitting of $E$ at
the boundary, is useful for treating the Hodge-Laplacian. See also
\cite{MMMT16, Taylor2}. See \cite{Taylor1} also for results on the
Robin problem on smooth domains. See \cite{Daners2000,
  gesztesyMitrea2009} for some results on the Robin problem on
non-smooth domains. See also \cite{AmariElliptic}.

\subsection{The bounded geometry of the boundary
  is needed}\label{ssec.saw}
  
Let us provide now an example of a manifold $\Omega$ with smooth
metric and smooth boundary $\pa \Omega$ that satisfies the Poincar\'e
inequality, and hence such that $\Delta \colon H^1_0(\Omega) \to
H^{-1}(\Omega)$ is an isomorphism, but such that $\Delta \colon
H^2(\Omega) \cap H^1_0(\Omega) \to L^2(\Omega)$ is not onto. In other
words, the operator $D = \Delta$ with Dirichlet boundary condition
does not satisfy an $H^2$-regularity estimate on $\Omega$, in spite of
the fact that it has coefficients in $W^{\infty, \infty}$. Our example
is based on the loss of regularity for problems on concave polygonal
domains.

Let $G \subset \RR^n$ be a bounded domain with \emph{Lipschitz}
boundary. It will be convenient to consider only \emph{closed}
domains. Let $m \in \NN$ (we shall need only the case $m=1$). We
denote by $H^m(G)$ the space of functions with $m$ derivatives in
$L^2$. It is the set of restrictions of functions from $H^m(\RR^n)$ to
$G$. We let $H^m_G(\RR^n)$ be the set of distributions in $H^m(\RR^n)$
with support in $G = \overline{G}$ and $H^m_0(G)$ be the closure of
the set of test functions with support in $G$ in $H^m(\RR^n)$, as
usual. It is known that $H^1_G(\RR^n) = H^1_0(G) = \ker(H^1(G) \to
L^2(G))$. See \cite{Taylor1} for the case $G$ smooth. The Lipschitz
case is completely similar and follows from $H^1_0(G) = \ker(H^1(G)
\to L^2(G))$, see, for example \cite{DorinaMitrea}.

Our construction of the manifold with boundary $\Omega$ is based on
the following lemma.

\begin{lemma}
Let $G$ and $\omega_1 \supset \omega_2 \supset \ldots \supset \omega_n
\supset \ldots$ be \emph{closed}, bounded domains in $\RR^n$ with
Lipschitz boundary such that $G := \cap \omega_n$. Let $f \in L^2(G)$,
$u_n \in H^1_0(\omega_n)$, and $w \in H^1_0(G)$ satisfy $\Delta u_n =
f$ on $\omega_n$ and $ \Delta w=f$ on $G$. If the domains $\omega_n$
have smooth boundary and $w \notin H^2(G)$, then
$\|u_n\|_{H^2(\omega_n)} \to \infty$.
\end{lemma}

\begin{proof}
We have that
\begin{equation*}
  H^1_0(\omega_1) \supset H^1_0(\omega_2) \supset \ldots \supset
  H^1_0(\omega_n) \ldots \supset H^1_0(G).
\end{equation*}
Let $\xi \in \cap_{n=1}^\infty H^1_0(\omega_n)$. Then $\xi$ has
support in $G = \cap_{n=1}^\infty \omega_n$, and hence $\xi$ has
support in $G$, which gives $\xi \in H^1_G(\RR^n) = H^1_0(G)$, by the
discussion preceding this lemma. This shows that $H^1(G) =
\cap_{n=1}^\infty H^1_0(\omega_n)$. Let $B(u, v) \define \int_{\RR^n}
(\nabla u, \nabla v) d\vol$. Then $B$ induces an inner product on
$H^1_0(\omega_1)$ equivalent to the initial inner product. Moreover,
the relations
\begin{equation*}
  B(u_n, v) = (f, v) = B(u_{n+1}, v), \quad v \in H^1_0(\omega_{n+1})
\end{equation*}
show that $u_{n+1}$ is the $B$-orthogonal projection of $u_n$ onto
$H^1_0(\omega_{n+1})$. Similarly, $w$ is the $B$-orthogonal projection
of $u_n$ onto $H^1_0(G)$. Since $H^1(G) = \cap_{n=1}^\infty
H^1_0(\omega_n)$, we obtain that $u_n \to w$ in $H^1_0(\omega_1)$.

To prove that $\|u_n\|_{H^2(\omega_n)} \to \infty$ if $w \notin
H^2(G)$, we shall proceed by contradiction. Let us assume then that
this is not the case. Then, by passing to a subsequence, we may assume
that $\|u_n\|_{H^2(\omega_n)}$ is bounded. Then $\|u_n \vert_G
\|_{H^2(G)} \le \|u_n\|_{H^2(\omega_n)}$ also forms a bounded
sequence. By passing to a subsequence again, we may assume then that
$u_n \vert_G$ converges weakly in $H^2(G)$ to some $\tilde w$, by the
Alaoglu-Bourbaki theorem. Hence $u_n\vert_G \to \tilde w$ weakly in
$H^1(G)$ (even in norm, since $H^2(G) \to H^1(G)$ is compact). We have
$\|u_n \vert_G - w\|_{H^1(G)} \le \|u_n - w\|_{H^1(\omega_1)} \to
0$. Hence $u_n\vert_G \to w$ in $H^1(G)$. Consequently, $\tilde w =
w$, which is a contradiction, since we have assumed that $w \notin
H^2(G)$.
\end{proof}

We are ready now to construct our manifold $\Omega$. Let $G$ be a
bounded domain whose boundary is smooth, except at one point, where we
have an angle $>\pi$ (so $G$ is \emph{not} convex). It is known then
that there exists $u \in H^1_0(G)$, $u \notin H^2(G)$ such that $f :=
\Delta u \in L^2(M)$. See \cite{DaugeBook, Grisvard1, Grisvard2,
  NP}. In fact, the space of functions $\phi \in L^2(W)$ such that
$\Delta^{-1} \phi \in H^2(G) \cap H^1_0(G)$ is of codimension one in
$L^2(G)$. Let $\ldots \subset \omega_{n+1} \subset \omega_n \subset
\ldots \subset \omega_1$ be a sequence of \emph{closed}, smooth,
bounded domains whose intersection is $G$. Then we can take for
$\Omega$ the (disjoint) union of all the domains $\omega_n \times
\{n\}$, $n \in \NN$.

Let us check that $\Omega$ has the desired properties. We construct
$\phi \in L^2(\Omega)$ by taking $\phi \define c_n f$ on $\omega_n$,
$c_n > 0$. Let $u_n \in H^1_0(\omega_n)$ be the unique solution of
$\Delta u_n = f$. We can choose $c_n$ such that $\sum_n c_n^2 <
\infty$, and hence $\phi \in L^2(\Omega)$, but $\sum_n c_n^2
\|u_n\|_{H^2(\omega_n)}^2 = \infty$, since the sequence
$\|u_n\|_{H^2(\omega_n)}$ is unbounded, by the last lemma. We have
that $\Omega \subset \omega_1 \times \NN$, and hence it satisfies the
Poincar\'e inequality for all functions in $H^1_0(\Omega)$. Let $u \in
H^1_0(\Omega)$ be the unique solution of the equation $\Delta u =
\phi$ \cite{AGN1}. Then $u = c_n u_n$ on $\omega_n$, by the uniqueness
of $u_n$, and hence $\|u\|^2_{H^2(\Omega)} = \sum_{n=1}^\infty c_n^2
\|u_n\|^2_{H^2(\omega_n)} = \infty$. That is, $u \notin H^2(\Omega)$.
Note that $\Omega$ is \emph{not} of bounded geometry: indeed, the
second fundamental form of $\omega_n$ cannot be uniformly bounded in
$n$, since the ``limit'' $G$ of the domains $\omega_n$ has a corner.

\def\cprime{$'$}

\end{document}